\documentclass[a4paper]{amsart}

\usepackage[margin=3cm]{geometry}

\usepackage{etex} 

\usepackage{amsmath}
\usepackage{amssymb}
\usepackage{amsthm}
\usepackage{amsfonts}
\usepackage[foot]{amsaddr}
\usepackage[usenames,dvipsnames,svgnames]{xcolor}
\usepackage{hyperref}
\usepackage{mathrsfs}
\usepackage{pictexwd, dcpic}
\usepackage{pinlabel}
\usepackage{stmaryrd}

\usepackage{tikz}
\usepackage{tikz-cd}
\usepackage{caption}
\usetikzlibrary{decorations.markings}

\numberwithin{equation}{section}

\newtheorem{main}{Theorem}
\newtheorem{thm}{Theorem}[section]
\newtheorem{lem}[thm]{Lemma}
\newtheorem{prop}[thm]{Proposition}
\newtheorem{cor}[thm]{Corollary}
\newtheorem{fact}[thm]{Fact}

\newtheorem*{meta}{Metaconjecture}

\newcommand{\bc}{{\bf c}}
\newcommand{\bu}{{\bf u}}
\newcommand{\bv}{{\bf v}}

\newcommand{\bx}{{\bf x}}
\newcommand{\by}{{\bf y}}
\newcommand{\bz}{{\bf z}}

\DeclareMathOperator{\Mod}{Mod^{\pm}}
\DeclareMathOperator{\mcg}{Mod}
\DeclareMathOperator{\SMod}{SMod^{\pm}}
\DeclareMathOperator{\Smcg}{SMod}
\DeclareMathOperator{\pmcg}{PMod}

\newcommand{\vR}{\mathcal{R}}
\newcommand{\vM}{\mathcal{M}}
\newcommand{\vG}{\mathcal{G}}
\newcommand{\vC}{\mathcal{C}}
\newcommand{\vN}{\mathcal{N}}
\newcommand{\vI}{\mathcal{I}}
\newcommand{\vSI}{\mathcal{SI}}
\newcommand{\vSN}{\mathcal{SN}}
\DeclareMathOperator{\Aut}{Aut}
\DeclareMathOperator{\Out}{Out}

\DeclareMathOperator{\Ex}{Ex}
\DeclareMathOperator{\Comm}{Comm}
\DeclareMathOperator{\Id}{Id}

\begin{document}
\title[Normal subgroups of braid groups]{Normal subgroups of the braid group and the metaconjecture of Ivanov}
\author{Alan McLeay}
\address{School of Mathematics and Statistics, University of Glasgow, Glasgow, G12 8QQ, Scotland}
\email{a.mcleay.1@research.gla.ac.uk}

\maketitle

\begin{abstract}
We show that many normal subgroups of the braid group modulo its centre, and of the mapping class group of a sphere with marked points, have the property that their automorphism and abstract commensurator groups are mapping class groups of such spheres. As one application, we establish the automorphism groups of each term in the lower central series and derived series of the pure braid group. We also obtain new proofs of results of Dyer-Grossman and Orevkov, showing that the automorphism groups of the braid group and of its commutator subgroup are isomorphic. We then calculate the automorphism and abstract commensurator groups of each term in the hyperelliptic Johnson filtration, recovering a result of Childers for the Torelli case. The techniques used in the paper rely on resolving a metaconjecture of Nikolai V. Ivanov for ``graphs of regions'', extending work of Brendle-Margalit.
\end{abstract}

\section{Introduction}\label{INTRO}

The \emph{braid group} on $n$-strands $B_n$ is generated by $\sigma_1, \dots, \sigma_{n-1}$ subject to the relations
\[
\sigma_i \sigma_j = \sigma_j \sigma_i \mbox{ for } |i-j|>1 \hspace{1cm} \mbox{ and } \hspace{1cm} \sigma_i \sigma_j \sigma_i = \sigma_j \sigma_i \sigma_j \mbox{ for } |i-j|=1.
\]
The latter are known as \emph{the braid relations}. The centre $Z$ of $B_n$ is generated by the element $(\sigma_1 \dots \sigma_{n-1})^n$. In this paper we give the automorphism group of any normal subgroup $N$ of $B_n$ such that $N \cap Z$ is trivial, providing $N$ contains an element that can be expressed product of at most one third of the generators given above. This result is given explicitly in Corollary \ref{BraidCOR}, the proof of which relies on an interpretation of $B_n/Z$ as a subgroup of the \emph{mapping class group} of a sphere with $n+1$ marked points.

\subsection*{Braid groups as mapping class groups.}
Let $\Sigma_{g,n}^b$ be an oriented surface of genus $g$ with $b$ boundary components and $n$ marked points. We can equivalently think of marked points as punctures. If the surface has no boundary components we write $\Sigma_{g,n}$. The focus of this paper will be spheres with marked points, that is, surfaces homeomorphic to $\Sigma_{0,n}$. We denote such a surface by $S_n$. The \emph{mapping class group} $\mcg (\Sigma_{g,n}^b)$ is the group whose elements are precisely all homotopy equivalence classes of orientation preserving self-homeomorphisms of the surface $\Sigma_{g,n}^b$ that fix the boundary pointwise. The \emph{extended mapping class group} $\Mod(\Sigma_{g,n})$ is the group of homotopy classes of all homeomorphisms, including the orientation-reversing ones.  The subgroup of $\mcg(\Sigma_{g,n}^b)$ consisting of all elements represented by homeomorphisms that fix a chosen marked point $p$ is denoted $\mcg(\Sigma_{g,n}^b,p)$.

Let $D_n$ be a disc with $n$ marked points, that is, a surface homeomorphic to $\Sigma_{0,n}^1$. One can define an isomorphism from $B_n$ to $\mcg(D_n)$ such that the image of each $\sigma_i$ is a half twist. By collapsing the boundary $\partial D_n$ to a point $p$, we see that $\mcg(S_{n+1},p)$ is isomorphic to $\mcg(D_n)$ modulo the subgroup generated by the Dehn twist about a curve isotopic $\partial D_n$. This Dehn twist generates the centre $Z$ of $B_n$ and so it follows that
\[
B_n / Z \cong \mcg(S_{n+1},p).
\]
Analogously to $\mcg(S_{n+1},p)$, we define the subgroup $\Mod(S_{n+1},p)$ of $\Mod(S_{n+1})$ consisting of all elements that fix the chosen marked point $p$. We can therefore define the following commutative diagram of inclusion maps:
\[
\begin{tikzcd}
& \mcg(S_{n+1},p) \arrow[hookrightarrow]{dr} \arrow[hookrightarrow]{dd} & \\
 \mcg(S_{n+1} ) \arrow[hookleftarrow]{ur} \arrow[hookrightarrow]{dr}&  & \Mod(S_{n+1},p) \\
 & \Mod(S_{n+1}) \arrow[hookleftarrow]{ur} & 
\end{tikzcd}
\]
Note that if $N$ is a subgroup of $\mcg(S_{n+1},p)$ that is normal in both $\mcg(S_{n+1})$ and $\Mod(S_{n+1},p)$ then it is also normal in $\Mod(S_{n+1})$. Furthermore, each subgroup in the diagram is of finite index. We define the injective homomorphism
\[
\Psi : B_n / Z \hookrightarrow \Mod(S_{n+1})
\]
by the isomorphism and inclusion maps discussed above. We use this interpretation to obtain the following result.

\begin{main}\label{BraidTHM}
Let $N$ be a normal subgroup of $B_n / Z$ containing an element represented by a product of at most $(n-1)/3$ standard generators. Then $\Aut N$ is isomorphic the normaliser of $\Psi(N)$ in $\Mod(S_{n+1})$.
\end{main}

We see then that $\Aut N$ is isomorphic to $\mcg(S_{n+1},p)$, $\mcg(S_{n+1})$, $\Mod(S_{n+1},p)$, or $\Mod(S_{n+1})$. As an example, we can apply Theorem \ref{BraidTHM} to the normal subgroup $\mathcal{BI}/Z$, where $\mathcal{BI}$ is the \emph{braid Torelli group} (see \cite{BMP}). It follows that $\Aut \mathcal{BI}/Z \cong \Mod(S_{n+1})$. Furthermore, this is also true for each of the congruence subgroups of $\mathcal{BI}$ modulo the centre $Z$.

We can also apply Theorem \ref{BraidTHM} to the group $B_n / Z$ itself. We see that $\Aut B_n / Z \cong \Mod(S_{n+1},p)$. In fact, this can be shown using the fact that $\mcg(S_{n+1},p)$ is a finite index subgroup of $\mcg(S_{n+1})$ as discussed in Section \ref{CurveGraph}. Considering the braid relations, and the fact that $Z \cong \mathbb{Z}$, one is able to prove that the natural homomorphism
\[
\Aut B_n \rightarrow \Aut B_n / Z
\]
is an isomorphism. The subgroup $\Mod(S_{n+1},p)$ can be generated by elements of $\mcg(S_{n+1},p)$ and a single orientation reversing element of $\Mod(S_{n+1},p)$ and so we therefore recover the following isomorphism of Dyer-Grossman \cite{DG}.
\begin{cor}\label{DG}
If $n \ge 4$ then
\[
\Aut B_n \cong \Mod(S_{n+1},p) \cong B_n / Z \rtimes \mathbb{Z} / 2 \mathbb{Z}.
\]
\end{cor}

If $N$ is any normal subgroup of $B_n$ such that $N \cap Z = \{1\}$ then $N$ is isomorphic to a normal subgroup of $B_n / Z$. It follows that Theorem \ref{BraidTHM} also gives the automorphism groups of subgroups of this type.

\begin{cor}\label{BraidCOR}
Let $N$ be a normal subgroup of $B_n$ containing an element represented by a product of at most $(n-1)/3$ standard generators. If $N \cap Z = \{1\}$ then $\Aut N$ is isomorphic to the normaliser of $\Psi(N)$ in $\Mod(S_{n+1})$.
\end{cor}

In particular, the commutator subgroup $[B_n, B_n]$ has trivial intersection with the centre, leading to the following result.
\begin{cor}\label{Commutator}
If $n \ge 7$ then
\[
\Aut [B_n,B_n] \cong \Mod(S_{n+1},p) \cong B_n / Z \rtimes \mathbb{Z} / 2 \mathbb{Z}.
\]
\end{cor}
This isomorphism was originally proved by Orevkov for $n \ge 4$ using more algebraic methods \cite{Orevkov}. We note that each term further down the lower central series and the derived series of the braid group is equal to $[B_n,B_n]$, see for example \cite{GorLin}. The pure braid group $PB_n$ is the kernel of the natural homomorphism from $B_n$ to the symmetric group on the set of $n$ elements. While $Z$ is also the centre of $PB_n$, it has trivial intersection with its commutator subgroup $[PB_n, PB_n]$. Since $PB_n$ is characteristic in $B_n$ we are able to establish the automorphism group of $[PB_n,PB_n]$ and in fact we arrive at a more general result.

\begin{cor}\label{PBG}
If $n \ge 7$ and $\Gamma$ is any term in the lower central series or derived series of $PB_n$ then
\[
\Aut \Gamma \cong \Aut PB_n / Z \cong \Mod(S_{n+1}).
\]
\end{cor}

Note that $[B_n,B_n]$, and every subgroup in the statement of Corollary \ref{PBG}, contain elements that can be written as a product of two standard generators. Therefore to apply Corollary \ref{BraidCOR} we require that $(n-1)/3 \ge 2$, that is, $n \ge 7$.

\subsection*{The hyperelliptic Johnson filtration.}
Let $\iota : \Sigma_{g,0} \rightarrow \Sigma_{g,0}$ be a fixed hyperelliptic involution of a genus $g>0$ surface with no marked points or boundary components. We abuse notation by writing $\iota$ for the element of $\mcg(\Sigma_{g,0})$ represented by this homeomorphism. Recall that $\Smcg(\Sigma_{g,0})$ is the \emph{hyperelliptic mapping class group} (or \emph{symmetric mapping class group}), that is, the subgroup of $\mcg(\Sigma_{g,0})$ consisting of all elements that commute with $\iota$.

Let $\Gamma_k$ denote the $k^{th}$ term in the lower central series of $\pi_1(\Sigma_{g,0})$. The \emph{Johnson filtration} of $\mcg(\Sigma_{g,0})$ is the sequence of groups $\vN_k(\Sigma_{g,0})$ defined to be the kernel of the homomorphisms
\[
\mcg(\Sigma_{g,0}) \rightarrow \Out(\Gamma / \Gamma_k).
\]
It has recently been shown by Brendle-Margalit for $g \ge 7$ that $\Aut \vN_k(\Sigma_{g,0}) \cong \Mod(\Sigma_{g,0})$ \cite{TD}. The first term in the Johnson filtration is the \emph{Torelli group} $\vI(\Sigma_{g,0})$, the subgroup of $\mcg(\Sigma_{g,0})$ consisting of elements acting trivially on homology. We define the \emph{hyperelliptic Torelli group} $\vSI(\Sigma_{g,0})$ to be the intersection $\Smcg(\Sigma_{g,0}) \cap \vI(\Sigma_{g,0})$. It is a result of Childers that
\[
\Aut \vSI(\Sigma_{g,0}) \cong \SMod(\Sigma_{g,0}) / \langle \iota \rangle,
\]
when $g \ge 3$ \cite{LRC}. Here, $\SMod(\Sigma_{g,0})$ is the \emph{extended hyperelliptic mapping class group}. Given these results we may ask the following natural question: what are the automorphism groups are of terms appearing further down the \emph{hyperelliptic Johnson filtration}? That is, the groups
\[
\vSN_k(\Sigma_{g,0}) = \Smcg(\Sigma_{g,0}) \cap \vN_k(\Sigma_{g,0}),
\]
for any $k$.  Recall that for any group $G$ we define $\Comm G$ to be the group of equivalence classes of isomorphisms between finite index subgroups of $G$, where two isomorphisms are equivalent if they agree on some finite index subgroup.

\begin{main}\label{HJF}
For all surfaces $\Sigma_{g,0}$ with $g \ge 6$ we have that
\[
\Comm \vSN_k(\Sigma_{g,0}) \cong \Aut \vSN_k(\Sigma_{g,0}) \cong \SMod(\Sigma_{g,0}) / \langle \iota \rangle.
\]
\end{main}
\noindent In fact, as we discuss later in this section, we can recover the better bound of $g \ge 3$ for the hyperelliptic Torelli group given by Childers \cite{LRC}.

\subsection{The metaconjecture of Ivanov}\label{CurveGraph}

The proofs of Theorem \ref{BraidTHM} and Theorem \ref{HJF} rely on the study of a wide class of graphs associated to spheres. A well known member of this class is the \emph{curve graph} of the sphere, defined by Harvey \cite{HAR}. The curve graph $\mathcal C (\Sigma_{g,n})$ has vertices corresponding to the isotopy classes of essential simple closed curves on the surface $\Sigma_{g,n}$. Two vertices are adjacent if they correspond to isotopy classes containing disjoint representatives.

For any surface $\Sigma_{g,n}$ such that $g \ge 2$, other than $\Sigma_{2,0}$, Ivanov proved that the natural homomorphism
\[
\eta : \Mod(\Sigma_{g,n}) \rightarrow \Aut \vC (\Sigma_{g,n})
\]
is an isomorphism \cite{Ivanov}. This result combined with work of Korkmaz \cite{MK} and Luo \cite{Luo} classified all surfaces for which this natural homomorphism is an isomorphism. In particular, Korkmaz proved this for spheres $S_n$ where $n \ge 5$. One of many applications of these results is that $\Aut \mcg (\Sigma_{g,n})$ is isomorphic to $\Mod (\Sigma_{g,n})$ for these surfaces.

A further implication is that each isomorphism between finite index subgroups of $\Mod(\Sigma_{g,n})$ is induced by a mapping class $f \in \Mod(\Sigma_{g,n})$. This fact has been used by Charney-Crisp to prove that the automorphism group of a finite index subgroup of $\Mod(S_n)$ is isomorphic to its normaliser in $\Mod(S_n)$ implying that the groups $\Aut B_n/Z$ and $\Mod(S_{n+1},p)$ are isomorphic \cite[Corollary 4 (2)]{ChCr}. The relationship between spheres with marked points and braid groups has also been studied by Leininger-Margalit \cite{LeinMarg} and Bell-Margalit \cite{BellMarg1}, \cite{BellMarg2}.

It has been shown that many other mathematical objects associated to surfaces besides the curve graph have automorphism group isomorphic to the extended mapping class group, such as the complex of non-separating curves \cite{Nonsep}, the complex of separating curves \cite{Sep}, the complex of domains \cite{MCP}, the truncated complex of domains \cite{MCP}, the arc complex \cite{Arc}, the arc and curve complex \cite{Arccurve}, and the complex of strongly separating curves \cite{Strongly}, the pants complex \cite{Pants}, the Torelli complex \cite{Kida}, and the ideal triangulation graph \cite{Triangulation} among others. Following these results Ivanov made a metaconjecture.

\begin{meta}[Ivanov \cite{Meta}]
Every object naturally associated to a surface $\Sigma_{g,n}$ and having a sufficiently rich structure has $\Mod ( \Sigma_{g,n} )$ as its group of automorphisms. Moreover, this can be proved by a reduction to the theorem about automorphisms of $\mathcal C (\Sigma_{g,n})$.
\end{meta}

Brendle-Margalit have defined a wide class of graphs upon which the group $\Mod (\Sigma_{g,0})$ acts naturally and resolved the metaconjecture for such graphs \cite{TD}. Indeed, they use these graphs to resolve the metaconjecture for many normal subgroups of $\Mod(\Sigma_{g,0})$. In this paper we extend this definition to graphs associated to spheres $S_n$ and resolve the metaconjecture for this class.

\subsection*{Regions and complementary discs.} A \emph{region} $R$ is a connected, compact subsurface of $S_n$ with at least one boundary component, where each such boundary component is an essential simple closed curve in $S_n$. We write $n(R)$ for the number of marked points in the region $R$. It follows that every region $R \subset S_n$ is homeomorphic to a surface $\Sigma_{0,n(R)}^b$ where $b>0$ and $n(R) \le n$. A \emph{complementary disc} of a region $R$ is a disc $D$ that is disjoint from $R$ and homotopic to a component of $S_n \setminus R$. Since the boundary components of a region are essential, it follows that each complementary disc contains at least two marked points.

We now define the set $\vR (S_n)$ consisting of the $\Mod (S_n)$-orbits of all regions of $S_n$. For any region $R$, if the $\Mod(S_n)$-orbit of $R$ belongs to $A \subset \vR(S_n)$ then we say that $R$ \emph{represents} an element of $A$, or that $R$ is \emph{represented} in $A$.

\subsection*{Graphs of regions.} Given a subset $A \subset \vR(S_n)$ we define the \emph{graph of regions}, written $\vG_A(S_n)$, to be a graph whose vertices correspond to the homotopy equivalence classes of regions that represent elements of $A$. Adjacency between two vertices occurs whenever the corresponding equivalence classes of regions contain disjoint representatives. We will say that a vertex $v$ \emph{corresponds} to a region $R$ if it corresponds to the equivalence class of $R$. The results in this paper can equivalently be stated in terms of the simplicial flag complexes admitted by graphs of regions.
\begin{figure}[t]
\centering
\includegraphics[scale=0.8]{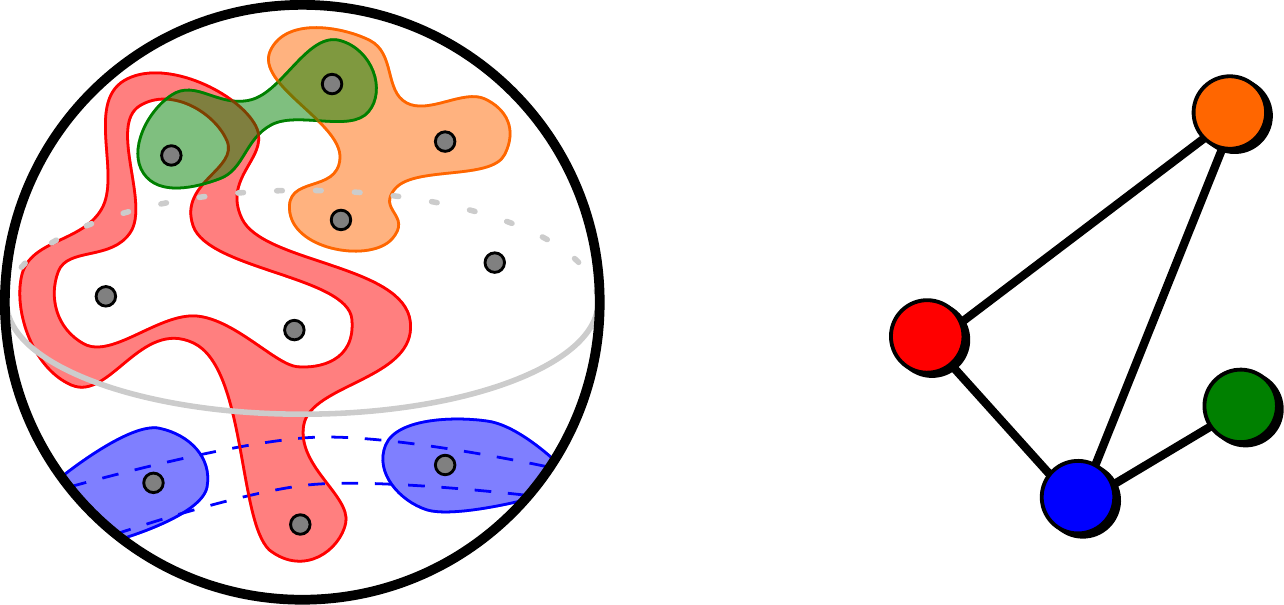}
\caption{Four vertices in a graph of regions.  Two vertices correspond to discs containing two marked points, one corresponds to a disc containing three marked points, and one vertex corresponds to an annulus containing one marked point bounding a disc containing three marked points.}
\label{GORfig}
\end{figure}
Note that if $A \subset \vR (S_n)$ is the set of annular subsurfaces of $S_n$ then $\vG_A(S_n) \cong \mathcal C (S_n)$. It follows then that the curve graph and many of its subgraphs can be interpreted as graphs of regions.

\subsection*{Exchange automorphisms.} Following McCarthy-Papadopoulos \cite{MCP}, we define $\phi \in \Aut \vG_A (S_n)$ to be an \emph{exchange automorphism} if there exist vertices $v_1, v_2 \in \vG_A (S_n)$ such that 
$\phi(v_1)=v_2$, $\phi(v_2)=v_1$ and $\phi(v)=v$ for all other vertices $v \not \in \{v_1, v_2\}$ of $\vG_{A} (S_n)$.

Brendle-Margalit showed for surfaces of the form $\Sigma_{g,0}$ that graphs of regions admitting exchange automorphisms are precisely those which fail to agree with the mataconjecture \cite[Theorem 2.2]{TD}.  This result carries over to spheres with marked points. We visit this fact in Section \ref{EA} using Brendle-Margalit's terminology of \emph{holes} and \emph{corks} which we now describe.

\subsection*{Holes and corks.}
We say that a vertex $v$ of $\vG_{A}(S_n)$ is a \emph{hole} if it corresponds to a region $R$ with complementary disc $D$ such that no subsurface of $D$ represents an element of $A$. Let $Q$ be a union of such complementary discs whose intersection with $R$ is a collection of annuli. The \emph{filling} of a hole $v$ is defined up to homotopy to be the region $R' = R \cup Q$.

We say a vertex $v$ is a \emph{cork} if it corresponds to an annulus $R$ with a complementary disc $D$ representing an element of $A$ such that the only subsurfaces of $D$ representing an element of $A$ are homotopic to $R$. If $u$ is the vertex corresponding to $D$ we call $u$ and $v$ a \emph{cork pair}.

\subsection*{Examples.}
We give an example of holes with equal fillings and cork pairs in Figure \ref{CorkHole}. Let $A \subset \vR (S_n)$ be the subset represented by all regions except discs containing two marked points and the annular boundaries of such discs. The two regions $R_1$ and $R_2$ in Figure \ref{CorkHole}(i) each have a complementary disc containing two marked points. The only regions that are contained in such discs are precisely those which are excluded from the subset $A$ and so both $R_1$ and $R_2$ correspond to holes in $\vG_{A}(S_n)$. The fillings of these holes are equal and homeomorphic to a disc with three marked points. Note that the filling of a hole is not necessarily a region represented in $A$.

Now, consider $B \subset \vR (S_n)$ to be the subset represented by all regions except discs containing two marked points. The regions $R_1$ and $R_2$ do not represent holes in $\vG_{B}(S_n)$ as their complementary discs contain regions homotopic to their annular boundaries, which are represented in $B$.

The annular region $R_1$ shown in Figure \ref{CorkHole}(ii) has a complementary disc $R_2$ containing three marked points. Since $R_2$ represents an element of $A$ and the only proper, non-peripheral regions $R_2$ are those which are excluded from $A$ it follows that $R_1$ corresponds to a cork in $\vG_{A}(S_n)$. Again, $R_1$ does not correspond to a cork in $\vG_{B}(S_n)$ as its complementary disc $R_2$ contains infinitely many annuli that represent elements of $B$.

\begin{figure}[t]
\centering
\labellist \hair 1pt
	\pinlabel {(i)} at 130 -7
	\pinlabel {(ii)} at 415 -7
    \endlabellist
\includegraphics[scale=.8]{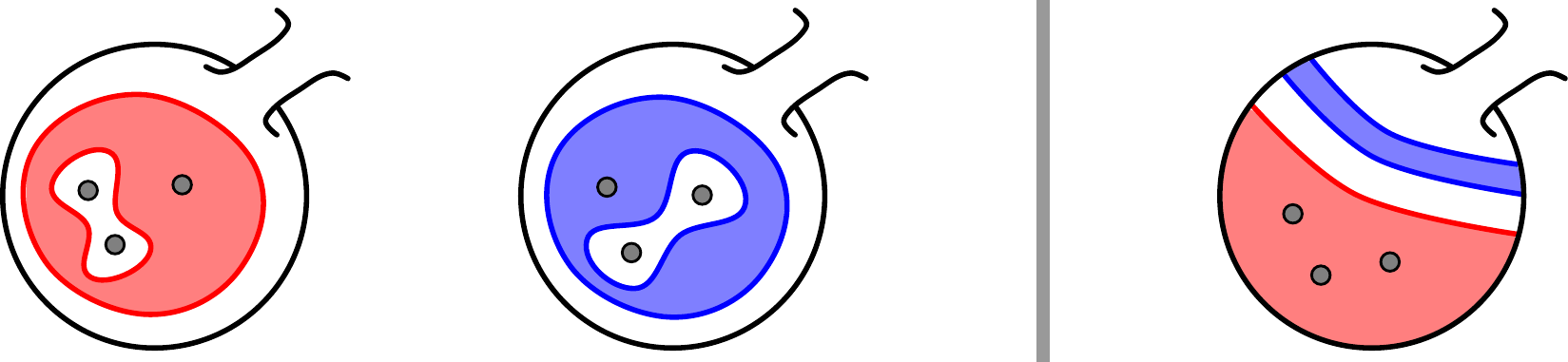}
\caption{(i) Vertices of $\vG_A(S_n)$ corresponding to these two regions are holes with equal fillings if no subsurface of $D_2$ is represented in $A$. (ii) Vertices of $\vG_A(S_n)$ corresponding to these two regions are a cork pair if no subsurface of $D_2$ is represented in $A$.}
\label{CorkHole}
\end{figure}

\subsection*{The natural homomorphism.} For any subset $A \subset \vR(S_n)$ there exists a natural homomorphism
\[
\eta_A \ :\ \Mod(S_n) \rightarrow \Aut \vG_{A}(S_n).
\]
Let $v$ be a vertex of $\vG_A(S_n)$ corresponding to a region $R$. For any $f \in \Mod(S_n)$ we define $\eta_A(f)(v) := u$ where $u$ is the vertex corresponding to the region $f(R)$.

\subsection*{Enveloping discs and smallness.}
An \emph{enveloping disc} $D_R$ of a region $R$ is a disc such that $R \subset D_R$ and for any other disc $D$ satisfying $R \subset D$ we have that $n(D_R) \le n(D)$.

Let $R$ be a region represented in some subset $A \subset \vR(S_n)$. We say that the vertex of $\vG_{A}(S_n)$ corresponding to $R$ is \emph{small} if $3 n(D_R) -1 \le n$. Similarly, we say a mapping class is \emph{small} in $\mcg(S_n) $if its support is a region $R$ such that $3 n(D_R) - 1 \le n$.  Finally, an element of $\mcg(S_n,p)$ is \emph{small} in $\mcg(S_n,p)$ if it is small in $\mcg(S_n)$ and its support is a region $R$ such that $p \notin D_R$.

\subsection{Results and applications.}\label{ResultsApps}
We now state the main theorem of this paper concerning graphs of regions associated to spheres.

\begin{main}\label{BigDaddy}
Let $A \subset \vR(S_n)$ and suppose its corresponding graph of regions $\vG_{A}(S_n)$ contains a small vertex. Then the natural homomorphism
\[
\eta_A : \Mod(S_n) \rightarrow \Aut \vG_{A}(S_n)
\]
is an isomorphism if and only if $\vG_{A}(S_n)$ has no holes and no corks.
\end{main}

In their paper, Brendle-Margalit also proved that if $\vG_{A}(\Sigma_{g,0})$ contains a small vertex then the automorphism group of the graph is isomorphic to the semidirect product of the $\Mod (\Sigma_{g,0})$ and the group generated by exchange automorphisms. This is also true in graphs $\vG_{A}(S_n)$ containing a small vertex. The techniques used in the proof of this fact are the same as \cite[Theorem 2.2]{TD} and so it is omitted from this paper.

Using the work of Brendle-Margalit \cite[Section 6]{TD} and Ivanov \cite{TeichModGroups} we then prove the following proposition.

\begin{main}\label{Proppy}
Let $N$ be a normal subgroup of either $\mcg(S_n)$ or $\mcg(S_n,p)$ such that $N$ contains a small element. Then $\Aut N$ is isomorphic to the normaliser of $N$ in $\Mod(S_n)$. Furthermore, if $N$ is normal in $\Mod(S_n)$ then the group of abstract commensurators $\Comm N$ is $\Mod(S_n)$.
\end{main}

Before we apply Theorem \ref{Proppy} to normal subgroups of braid groups and the hyperelliptic Johnson filtration it is useful to consider the following subgroup of $\mcg(S_n)$.  Recall that there is an exact sequence
\[
1 \rightarrow P \rightarrow \mcg(S_n,p) \xrightarrow[]{F} \mcg(S_{n-1}) \rightarrow 1,
\]
given by Birman, where $F$ is the surjective homomorphism that forgets the marked point $p$.  We note that $P$ is not normal in $\mcg(S_n)$ yet contains an element that is a small element of $\mcg(S_n)$. Conversely, $P$ is normal in $\mcg(S_n,p)$ but every element has support containing $p$. Recall from the definition of small that no such element is small in $P$. We conclude therefore that Theorem \ref{Proppy} does not apply in this case. Indeed, the subgroup $P$ is isomorphic to a free group, implying that $\Aut P$ is considerably larger than its normaliser in $\Mod(S_n)$.

\subsection*{Application to braid groups}
To prove Theorem \ref{BraidTHM} we note that a small element of $\mcg(S_{n+1},p)$ has support contained in a disc containing $(n+2)/3$ marked points. We can therefore express this element as a product of at most $(n-1)/3$ half twists. If $N$ is a normal subgroup of $\mcg(S_{n+1},p)$ this implies the corresponding subgroup of $B_n / Z$ contains an element represented by a product of at most $(n-1)/3$ standard generators. Applying Theorem \ref{Proppy} completes the proof of Theorem \ref{BraidTHM}.

\subsection*{Application to the hyperelliptic Johnson filtration}
We have from Birman-Hilden \cite{BH} that
\[
\mcg(S_{2g+2}) \cong \Smcg(\Sigma_{g,0}) / \langle \iota \rangle.
\]
Since the hyperelliptic involution $\iota$ does not act trivially on homology we have that $\mathcal{SI}(\Sigma_{g,0})$ is isomorphic to a normal subgroup of $\mcg(S_{2g+2})$. We note that $\mathcal{SI}(\Sigma_{g,0})$ is generated by Dehn twists about symmetric separating curves of $\Sigma_g$ (separating curves that are fixed by $\iota$) as shown by Brendle-Margalit-Putman \cite{BMP}. Under the isomorphism these generators map to squares of Dehn twists about curves separating an odd number of marked points in $S_{2g+2}$. If $g\ge3$ then it follows that the image of $\mathcal{SI}(\Sigma_{g,0})$ in $\mcg(S_{2g+2})$ contains a small element. The normaliser of this subgroup is $\Mod(S_{2g+2})$ and so from Theorem \ref{Proppy} we recover the result of Leah Childers \cite{LRC} that
\[
\Comm \mathcal{SI}(\Sigma_{g,0}) \cong \Aut \mathcal{SI}(\Sigma_{g,0}) \cong \Mod(S_{2g+2}) \cong \SMod(\Sigma_{g,0}) / \langle \iota \rangle.
\]
Similarly, any term $\vSN_k(\Sigma_{g,0})$ is isomorphic to a normal subgroup of $\Mod(S_{2g+2})$. Using a construction similar to \cite[Proof of Theorem 5.10]{Meta} we can find an element of $\vSN_k(\Sigma_{g,0})$ whose support is contained in a genus two subsurface with one boundary component. This then corresponds to an element of $\Mod(S_{2g+2})$ whose support is contained in a disc with five marked points. From Theorem \ref{Proppy} and the definition of small we have that if $2g+2 \ge 3 (5) - 1$ then
\[
\Comm \vSN_k(\Sigma_{g,0}) \cong \Aut \vSN_k(\Sigma_{g,0}) \cong \Mod(S_{2g+2}) \cong \SMod(\Sigma_{g,0}) / \langle \iota \rangle.
\]
This is precisely the statement of Theorem \ref{HJF}.

\subsection*{Outline of the paper.}
In Section \ref{EA} we visit some results of Brendle-Margalit regarding exchange automorphisms and prove the injectivity part of Theorem \ref{BigDaddy}. The structure of the remainder of the paper centres around a study of the automorphisms of certain subgraphs of the curve graph. First, note that each curve $\bc$ in $S_n$ separates the surface into two discs, say $D_l$ and $D_m$ where $l \le m$. We call $\bc$ an $l$-curve and we call the discs $D_l$ and $D_m$ the small and large \emph{associated discs} of the curve $\bc$ respectively.
If a vertex $v$ of the curve graph $\mathcal C (S_n)$ corresponds to an $l$-curve then we call it an $l$-vertex.

We define $\vC_{k}(S_n)$ to be the subgraph of $\mathcal C (S_n)$ spanned by all $l$-vertices for which $l \ge k$. The proof of Theorem \ref{BigDaddy} relies on the following result which we will prove in Section \ref{Subgraphs}.

\begin{thm}\label{etas}
Let $S_n$ be a surface such that $n \ge 3 k - 1$. The natural homomorphism
\begin{center}$\eta_k\ :\ \Mod(S_n) \rightarrow \Aut \vC_{k} (S_n)$\end{center}
is an isomorphism.
\end{thm}

The proof follows an inductive argument where the base case is a theorem of Korkmaz discussed above, that is, we use the fact that $\vC_2(S_n) = \vC (S_n)$ and that $\Aut \vC(S_n) \cong \Mod(S_n)$ \cite{MK} . In Section \ref{GOR} we complete the proof of Theorem \ref{BigDaddy} by relating $\Aut \vG_{A}(S_n)$ with $\Aut \vC_{k}(S_n)$ for some $k$. Finally, we dedicate Section \ref{Subbies} to the proof of Theorem \ref{Proppy}.  As mentioned above, this result implies both Theorem \ref{BraidTHM} and Theorem \ref{HJF}.

\subsection*{Acknowledgments.} The author would like to thank his supervisor, Tara Brendle, for her helpful guidance and support. He is grateful to Dan Margalit for several helpful discussions and suggestions that greatly improved the paper. He is also grateful to Tyrone Ghaswala and Luis Paris for detailed comments on an earlier draft. Finally, he would like to thank Javier Aramayona, Vaibhav Gadre, and Shane Scott for their support and helpful discussions about the paper.

\section{Exchange Automorphisms and Injectivity}\label{EA}
In this section we give sufficient conditions for a graph of regions $\vG_{A}(S_n)$ to admit exchange automorphisms. We then prove that the natural homomorphism defined in Section \ref{INTRO} is injective.

Recall the example of holes given in Section \ref{INTRO} whose fillings are homotopic discs containing three marked points, see Figure \ref{CorkHole}(i). Label the two holes $u,v \in \vG_{A}(S_n)$ and write $D$ for one of these discs. A vertex is adjacent to $u$ or $v$ precisely when it corresponds to a region that is disjoint from $D$. It is clear, therefore, that $u$ and $v$ have equal links. We can then define an automorphism of the graph $\vG_{A}(S_n)$ that exchanges the two vertices $u$ and $v$. No such automorphism can belong to the image of the natural homomorphism $\eta_A$. Indeed, there is no element of $\Mod(S_n)$ that swaps these two regions.

Similarly, recall the example of the cork pair from Figure \ref{CorkHole}(ii). Here, any region that is disjoint from the disc $R_2$ is homotopic to a region that is disjoint from the annulus $R_1$, including the annulus itself. Furthermore, any region that is disjoint from $R_1$ is either disjoint from $R_2$, homotopic to $R_2$, or is a proper non-peripheral subsurface of $R_2$. Since there are no such subsurfaces represented in the subset $A \subset \vR (S_n)$ we conclude that the vertices $u,v \in \vG_{A}(S_n)$ corresponding to $R_1,R_2$ have equal stars. We can once again define an automorphsim exchanging $u$ and $v$. Clearly there is no element of $\Mod(S_n)$ that maps a disc to an annulus. We again conclude that the natural homomorphism $\eta_A$ is not surjective.

We give a version of \cite[Theorem 2.1]{TD} for surfaces $S_n$ with $n > 0$. The proof for spheres follows Brendle-Margalit's proof for surfaces of the form $\Sigma_{g,0}$. Indeed, the result follows from the definition of a graph of regions $\vG_A(S_n)$, complementary discs, and the proof of \cite[Theorem 2.1]{TD}.

\begin{lem}[Brendle-Margalit]\label{changies}
Let $S_n$ be a sphere with $n>0$. If $\vG_{A}(S_n)$ is a connected complex of regions then $\vG_{A}(S_n)$ admits exchange automorphisms if and only if it has hole or a cork. Moreover, two vertices admit an exchange automorphism if and only if they have holes with equal fillings or they form a cork pair.
\end{lem}

We now describe the automorphism group of a complex containing holes or corks, that is, a complex admitting exchange automorphisms. First, we define $\Ex \vG_{A}(S_n)$ to be the subgroup generated by the exchange automorphisms of $\vG_A(S_n)$.

\begin{lem}[Brendle-Margalit]
Let $A \subset \vR(S_n)$ and let $\vG_{A}(S_n)$ contain a small vertex. Then
\[
\Aut \vG_{A}(S_n) \cong \Ex \vG_{A}(S_n) \rtimes \Mod(S_n).
\]
\end{lem}

The proof for the analogous result for surfaces of the form $\Sigma_{g,0}$ can be found in \cite[Theorem 2.2]{TD}, and as with Lemma \ref{changies} relies solely on the definitions of corks, holes and graphs of regions and carries over to spheres with marked points.

We now prove that the natural homomorphism is injective. Note that we can apply this result to certain subgraphs of the curve graph by invoking the bijective relationship between the set of isotopy classes of curves and the set of homotopy classes of annuli.

\begin{lem}\label{Injectivity}
Let $A \subset \vR(S_n)$ and let $\vG_{A}(S_n)$ be a graph of regions. The natural homomorphism
\[
\eta_A : \Mod(S_n) \rightarrow \Aut \vG_{A}(S_n)
\]
is injective.
\end{lem}

\begin{proof}
Let $\bc$ be a $2$-curve. The large associated disc $D$ of $\bc$ is filled by regions that are represented in $A$. Equivalently, there exist regions represented in $A$ whose boundary components are curves that fill $D$. It follows then that if $f \in \Mod(S_n)$ is in the kernel of $\eta_A$ it must also fix $\bc$. Since our choice of $\bc$ was arbitrary we can find a pants decomposition $P$ such that $f$ fixes every curve in $P$. We conclude that $f$ is a product of Dehn twists and half twists and is therefore orientation preserving, hence $f \in \mcg(S_n)$. If $H_\bc$ is the half twist defined by an arbitrary $2$-curve $\bc$ then we have that
\[
H_\bc = H_{f(\bc)} = f H_\bc f^{-1},
\]
that is, $f$ and $H_\bc$ commute. Since our choice of $\bc$ was arbitrary and $\mcg(S_n)$ is generated by half twists we have that $f$ is in the centre of $\mcg(S_n)$. The centre of $\mcg(S_n)$ is trivial and so $\eta_A$ is injective.
\end{proof}

\section{The Curve Graph and its Subgraphs}\label{Subgraphs}

Recall that a vertex $v \in \vC (S_n)$ is an $l$-vertex if it corresponds to a curve whose small associated disc contains exactly $l\ge2$ marked points. We defined $\vC_k(S_n)$ to be the full subgraph of $\vC (S_n)$ spanned by $l$-vertices, where $l \ge k$. In this section we will prove Theorem \ref{etas}, that the natural homomorphism
\[
\eta_k \ :\ \Mod(S_n) \rightarrow \Aut \vC_{k}(S_n)
\]
is an isomorphism when $n \ge 3k - 1$.

Given a vertex $v \in \vC_{k}(S_n)$ corresponding to a curve $\bv$ we say that vertices $u$ and $w$ lie on the same side of $v$ if they correspond to curves that lie on the same associated disc of $\bv$. It is easy to see that $u$ and $w$ lie on the same side of $v$ if and only if there exists a vertex spanning an edge with $v$ but with neither $u$ nor $w$.

\begin{lem}\label{vertex}
If $v$ is a $l$-vertex then $\phi(v)$ is an $l$-vertex for all $\phi \in \Aut \vC_{k}(S_n)$.
\end{lem}

\begin{proof}
Let $\sigma$ be a maximal simplex of $\vC_{k}(S_n)$ containing $v$. The vertex $v$ is a $k$-vertex if and only if every vertex in $\sigma$ distinct from $v$ lies on the same side of $v$. Similarly we see that $v$ is a $(k+1)$-vertex if and only  every vertex on one side of $v$ is a $k$-vertex. We can continue the argument exhaustively, proving the lemma.
\end{proof}

We say that a vertex $u$ lies on the \emph{small side} of a vertex $v$ if it does not lie on the same side as some $\lfloor n/2 \rfloor$-vertex. That is, if $u$ corresponds to a curve that is contained in the small associated disc of a curve $\bv$, where $v$ corresponds to $\bv$.

\subsection*{Marked point sharing pairs.} Let $u$ and $v$ be two $k$-vertices of $\vC_k(S_n)$.  We say that $u$ and $v$ form a \emph{marked point sharing pair} if they correspond to curves with geometric intersection number two, one of the four subsurfaces obtained by cutting $S_n$ along the curves is a disc $D_{k-1}$ containing precisely $k-1$ marked points, and two of the subsurfaces are homeomorphic to a disc with a single marked point.

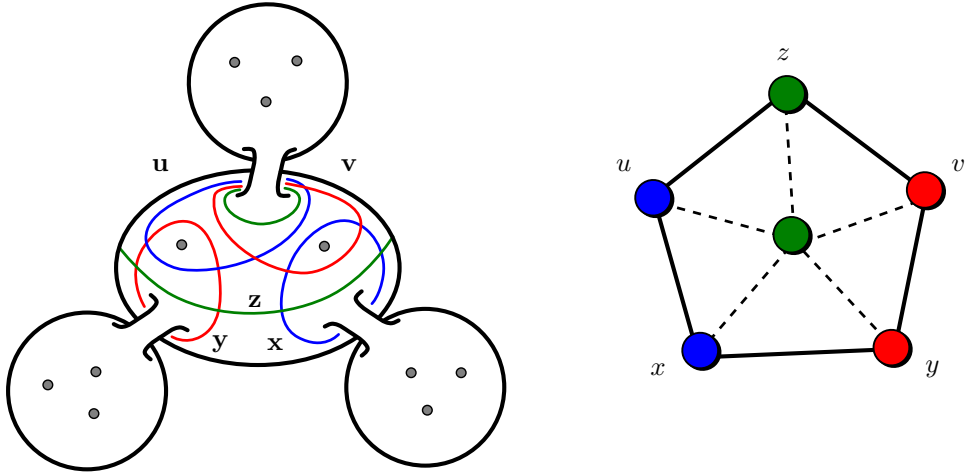
\begin{figure}[t]
\centering
\labellist \hair 1pt
	\pinlabel {$\bu$} at 90 180
	\pinlabel {$\bv$} at 200 180
	\pinlabel {$\bz$} at 145 100
	\pinlabel {$\by$} at 125 75
	\pinlabel {$\bx$} at 157 75
	\pinlabel {$u$} at 360 180
	\pinlabel {$v$} at 555 180
	\pinlabel {$z$} at 453 245
	\pinlabel {$x$} at 380 60
	\pinlabel {$y$} at 540 60
    \endlabellist
\includegraphics[scale=.65]{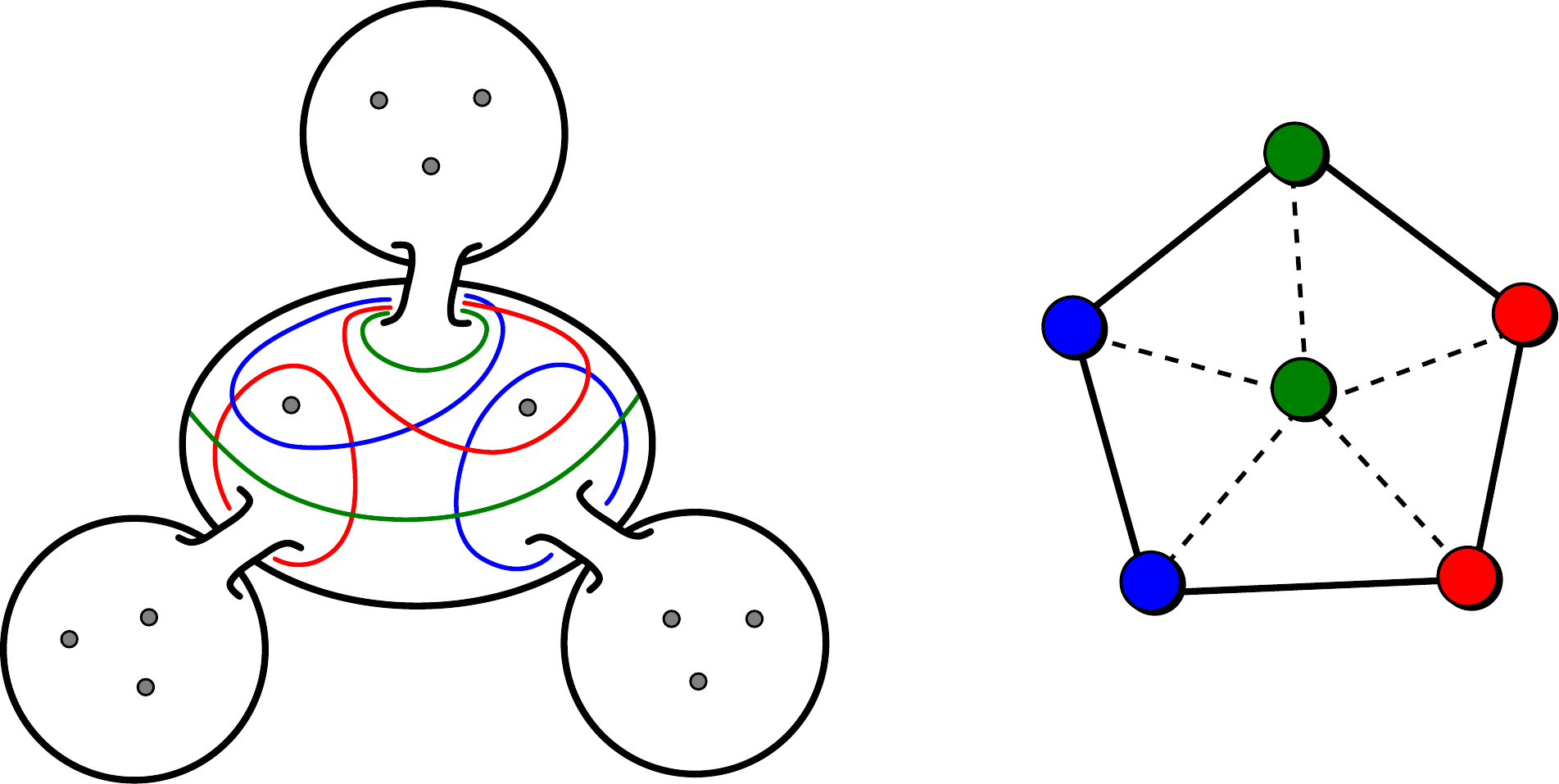}
\caption{A marked point sharing pair consisting of two $4$-vertices $u$ and $v$.}
\label{SharingPair}
\end{figure}

If a marked point sharing pair $u,v$ corresponds to the curves $\bu,\bv$ we say that $\bu$ and $\bv$ share the curve $\bc$, where $\bc$ is the $(k-1)$-curve isotopic to the boundary curve $\partial D_{k-1}$.

\begin{lem}\label{mp}
Let $n \ge 3k - 1$, and let $u$ and $v$ be two vertices of $\vC_{k}(S_n)$ that form a marked point sharing pair. For any $\phi \in \Aut \vC_{k}(S_n)$ the vertices $\phi(u)$ and $\phi(v)$ also form a marked point sharing pair.
\end{lem}

\begin{proof}
Let $u,v$ be two $k$-vertices of $\in \vC_{k}(S_n)$ with representative curves $\bu,\bv$. We will show that $u,v$ form a marked point sharing pair if and only if there are vertices $x,y,z$ of $\vC_{k}(S_n)$ with the following properties:
\begin{enumerate}
\item Both $u$ and $v$ lie on the small side of the $(k+1)$-vertex $z$;
\item the $k$-vertex $x$ is adjacent to $u$ and $y$ but not $v$; and
\item the $k$-vertex $y$ is adjacent to $v$ and $x$ but not $u$.
\end{enumerate}
First, let $u,v$ be a marked point sharing pair corresponding to curves $\bu,\bv$ with geometric intersection number two. Up to homeomorphism there is a unique configuration for curves $\bu,\bv$ shown in Figure \ref{SharingPair}. The curves $\bu$ and $\bv$ separate $S_n$ into four discs which are homeomorphic to $D_{k-1}$, $D_1$, $D_1$ and $D_{n-k-1}$. Let $\bz$ be the boundary of this final disc. We define $\bx$ and $\by$ to be the curves as shown in Figure \ref{SharingPair}. If we define $x,y,z$ to be the vertices of $\vC_{k}(S_n)$ corresponding to the curves $\bx,\by,\bz$ we see that they satisfy the three conditions above.

Now suppose we have vertices $u,v,x,y$ and $z$ that satisfy the above conditions. The vertex $z$ corresponds to a curve $\bz$ which bounds a disc $D_z$ (with $k+1$ marked points). The vertices $u$ and $v$ correpond to curves contained in the disc $D_z$. The second condition tells us that $x$ corresponds to a curve $\bx$ such that the intersection of $\bx$ with $D_z$ is a set arcs. Each arc in this set is isotopic to some $\bx'$ which lies in a subsurface $Q_x$ homeomorphic to an annulus with a single marked point. The arc $\bx'$ has endpoints on $\bz$, one of the two boundary components of $Q_x$. We want to show that the vertex $u$ corresponds to the other boundary component of $Q_x$.

Given any arc in $Q_x$ with both endpoints in $\bz$, the surface obtained by cutting $Q_x$ along the arc is a disjoint union of an annulus and a disc with one marked point. A boundary component $\bu$ of this annulus is isotopic to the boundary component of $Q_x$ not isotopic to $\bz$. It must therefore be a curve in $D_z$ bounding a disc containing $k$ marked points disjoint from $\bx$. It follows that the curve $\bu$ represents the vertex $u$.

By symmetry, the vertex $v$ corresponds to the boundary component of the equivalent subsurface $Q_y$ which is not isotopic to $\bz$. Define the curve $\by$ and the arc $\by'$ analogously to $\bx$ and $\bx'$. From the second and third conditions, the intersection of $\bx$ and $\by$ with $\bz$ are points as shown in Figure \ref{discs}(i). There exist segments $\gamma_x$ and $\gamma_y$ of $\bz$, with $\gamma_x \cup \gamma_y = \bz$, such that the arc $\bx'$ has endpoints in $\gamma_x$ and the arc $\by'$ has endpoints in $\gamma_y$.

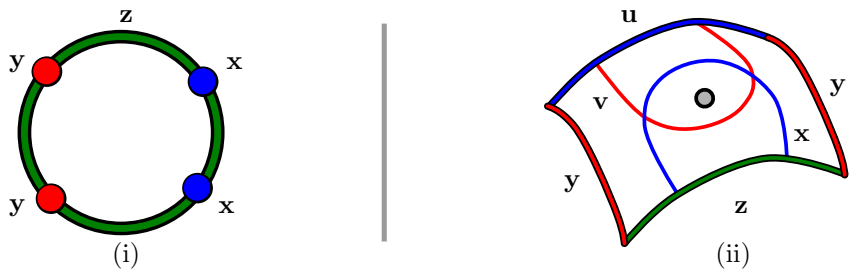
\begin{figure}[h]
\centering
\labellist \hair 1pt
	\pinlabel {(i)} at 135 -10
	\pinlabel {(ii)} at 890 -10
	\pinlabel {$\by$} at 0 230
	\pinlabel {$\by$} at 0 50
	\pinlabel {$\bz$} at 135 290
	\pinlabel {$\bx$} at 270 230
	\pinlabel {$\bx$} at 260 50
	\pinlabel {$\by$} at 690 80
	\pinlabel {$\bz$} at 900 50
	\pinlabel {$\by$} at 1020 200
	\pinlabel {$\bu$} at 760 290
	\pinlabel {$\bv$} at 725 180
	\pinlabel {$\bx$} at 975 140
    \endlabellist
\includegraphics[scale=.3]{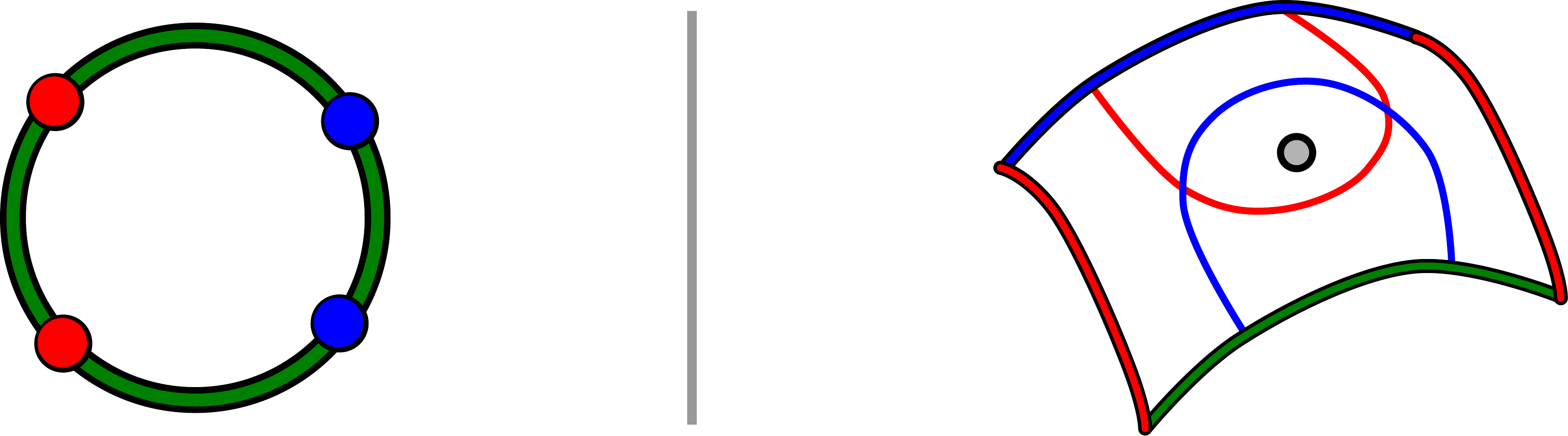}
\caption{(i) The curves $\bx$ and $\by$ intersect $\bz$ as shown.  (ii) The intersection of the curves $\bu$ and $\bv$ is restricted to the boundary of a disc containing a single marked point.}
\label{discs}
\end{figure} 

It follows that the intersection of the arc $\by'$ with $Q_x$ is a set of two isotopic arcs. If we cut along one of these arcs then since $\bu$ intersects $\bv$ they must take the form shown in Figure \ref{discs}(ii) where they intersect exactly twice.
If two simple closed curves in $D_z$ intersect in two points then they divide $D_z$ into four regions, one of which contains $\bz$. Given $u$ and $v$ are both $k$-vertices and $z$ is a $k+1$-vertex, it must be that two regions are homeomorphic to $D_1$, one is homeomorphic to $D_{k-1}$ and the region containing $\bz$ is an annulus. This satisfies the definition of a marked point sharing pair.
\end{proof}

We now move on to show that for a given automorphism $\phi \in \Aut \vC_{k}(S_n)$ we can define a unique extension $\widehat \phi \in \Aut \vC_{k-1}(S_n)$. We do this by first introducing the \emph{graph of sharing pairs} and showing that it consists of infinitely many connected components, each corresponding to a $(k-1)$-vertex of $\vC_{k-1}(S_n)$.

\subsection*{Graph of Sharing Pairs.} We call three $k$-vertices of $\vC_k(S_n)$ $u,v,w$ a \emph{sharing triple} if pairwise they form marked point sharing pairs and correspond to $k$-curves that each share the same $(k-1)$-curve $\bc$. We construct a graph $\mathcal{SP}$ with vertices corresponding to sharing pairs. Two vertices share an edge in $\mathcal{SP}$ if they correspond to sharing pairs $u,v$ and $v,w$ such that $u,v,w$ form a sharing triple. It is clear that sharing triples are characteristic in the graph $\vC_{k}(S_n)$ since sharing pairs are characteristic and the three vertices $u,v$, and $w$ must also lie on the small side of a unique $(k+2)$-vertex.

If two vertices are connected in $\mathcal{SP}$ then they correspond to $k$-curves that share the same $(k-1)$-curve. This implies that $\mathcal{SP}$ is made up of various disconnected components. We will write $\mathcal{SP}(\bc)$ for the components relating to sharing pairs corresponding to curves that share the curve $\bc$.

\begin{lem}\label{ShTr}
Let $S_n$ be a sphere with $n \ge 3k-1$. Let $\mathcal{SP}$ be the graph of sharing pairs associated to $\vC_{k}(S_n)$ and let $\bc$ be a curve shared by a sharing pair. The subgraph $\mathcal{SP}(\bc)$ is a single connected component of $\mathcal{SP}$.
\end{lem}

The proof of Lemma \ref{ShTr} makes use of the following result of Andrew Putman \cite[Lemma 2.1]{PutmanConnect}.

\begin{lem}[Putman]\label{Put}
Let $G$ be a group acting on a simplicial complex $X$ with $v$ a fixed vertex in $X^0$. Let $H$ be a set of generators of $G$ and such that:
\begin{enumerate}
\item for all $u \in X^0$, the orbit $G \cdot v$ intersects the connected component of $X$ containing $u$; and
\item for all $h \in H^{\pm 1}$, there is a path $P_h$ in $X$ from $v$ to $h \cdot v$.
\end{enumerate}
Then $X$ is connected.
\end{lem}

\begin{proof}[Proof of Lemma \ref{ShTr}]
Let $u,v$ be a marked point sharing pair corresponding to the curves $\bu,\bv$ that share the $k-1$-curve $\bc$. Let $D_{k-1}$ be the small associated disc of the curve $\bc$. Let $\Mod(S_n, D_{k-1})$ be the subgroup of $\Mod(S_n)$ that fixes the disc $D_{k-1}$. All curves that share $\bc$ must be of the form $f(\bu), f(\bv)$ for some $f \in \Mod(S_n, D_{k-1})$. This satisfies the first condition in Lemma \ref{Put} for the graph $\mathcal{SP}(\bc)$. It remains to show that the second condition is satisfied.
\begin{figure}[t]
\centering
\labellist \hair 1pt
	\pinlabel {$H(\bv)$} at 175 93
	\pinlabel {$\bv$} at 140 35
	\pinlabel {$\bu$} at 75 15
    \endlabellist
\includegraphics[scale=.6]{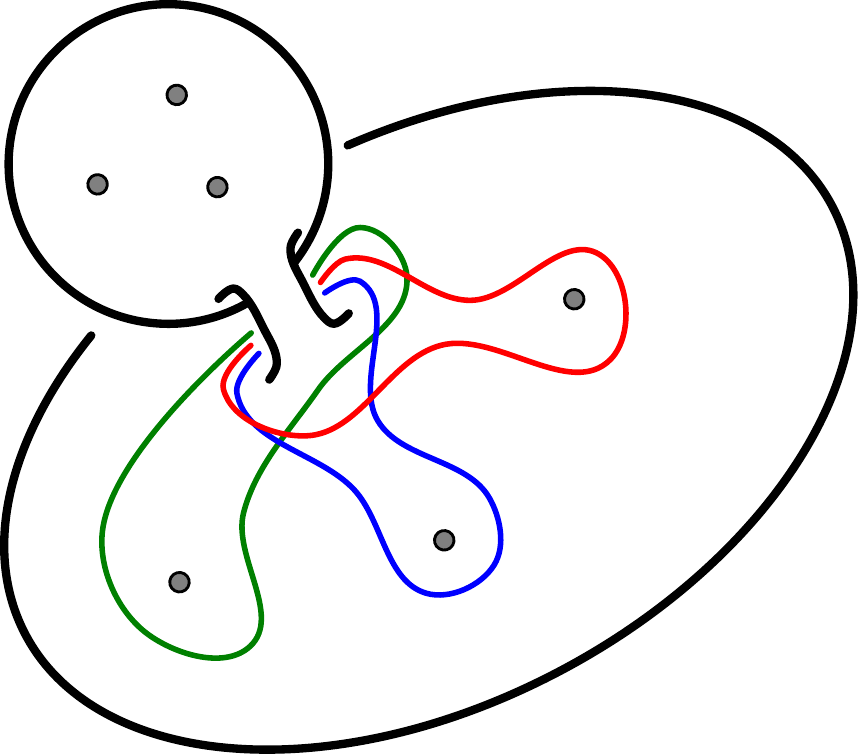}
\caption{The vertices corresponding to the sharing pairs $\bu, \bv$ and $H(\bu), H(\bv)$ share an edge in $\mathcal{SP}$ as $\bu, \bv, H(\bv)$ define a sharing triple.}
\label{SharingTriple}
\end{figure}

We can find a finite generating set for $\Mod(S_n, D_{k-1})$ consisting of half twists about $2$-curves. We can choose twists such that only one such curve intersects $\bu$ and a one curve intersects $\bv$. By Lemma \ref{Put} and symmetry we need only show that the vertex of $\mathcal{SP}(\bc)$ corresponding to the curves $\bu,\bv$ is connected to the vertex corresponding to the curves $\bu, H(\bu)$, where $H$ is a half twist fixing the curve $\bv$. As shown in Figure \ref{SharingTriple} the vertices share an edge.
\end{proof}

We now prove Theorem \ref{etas} which states that the homomorphism $\eta_k$ is an isomorphism. Note that the subgraph $\vC_{2}(S_n)$ is precisely the curve graph $\vC(S_n)$ itself. We use the result of Korkmaz showing that the natural homomorphism $\Mod(S_n) \rightarrow \Aut \vC_{2} (S_n)$ is an isomorphism \cite{MK}.

\begin{proof}[Proof of Theorem \ref{etas}]
By Lemma \ref{Injectivity} we have that $\eta_k$ is injective. As noted above, we have that $\eta_2$ is an isomorphism. Now we will use induction to show that if $\eta_{k-1}$ is an isomorphism then $\eta_k$ is surjective, hence an isomorphism. Let $\phi \in \Aut \vC_{k}(S_n)$ be an automorphism.  By Lemma \ref{ShTr} there exists a well defined automorphism $\widehat \phi$ of the vertices of $\vC_{k-1}(S_n)$ such that $\widehat \phi$ restricts to $\phi$. We will show that $\widehat \phi$ in fact extends to an automorphism of $\vC_{k-1}(S_n)$.

Suppose $u$ and $v$ are vertices of $\vC_{k-1}(S_n)$. If both vertices are also in $\vC_{k}(S_n)$ then this is clear. If $v$ is a vertex of $\vC_{k}(S_n)$ and $u$ is not then they are adjacent if and only if there exists some vertex $w$ adjacent to $v$ in $\vC_{k}(S_n)$ which contributes to a marked point sharing pair of $u$. Suppose now that neither $u$ nor $v$ are vertices of $\vC_{k}(S_n)$, then they are adjacent if and only if there are adjacent vertices $w_1$ and $w_2$ in $\vC_{k}(S_n)$ such that $w_1$ contributes to a sharing pair of $u$, and $w_2$ contributes to a sharing pair of $v$.

By assumption there exists some $f \in \Mod(S_n)$ whose image in $\Aut \vC_{k-1}(S_n)$ is precisely $\widehat \phi$. Since the restriction of $\widehat \phi$ to $\vC_{k}(S_n)$ is $\phi$ it follows that the image of $f$ in $\Aut \vC_{k}(S_n)$ is indeed $\phi$.
\end{proof}

\section{Graphs of Regions}\label{GOR}

Recall that the enveloping disc $D_R$ of a region $R$ is such that if any disc $D$ sastisfies $R \subset D$ then $n(D_R) \le n(D)$. We now fix the value of $k$ with respect to a subset $A \subset \vR(S_n)$ as follows:
\[
k = \min\{ n(D_R)\ |\ R \text{ is represented in } A\}.
\]
\noindent In this section we will construct a homomorphism from $\Aut \vG_{A}(S_n)$ to $\Aut \vC_{k}(S_n)$. We will then use this homomorphism and Theorem \ref{etas} to prove Theorem \ref{BigDaddy}.

The first step is defining the set map
\[
\Phi \ :\ \{ \text{vertices of }\vC_{k}(S_n) \} \rightarrow \{ \text{subgraphs of }\vG_{A}(S_n)\}.
\]
For any vertex $v \in \vC_{k}(S_n)$ corresponding to a curve $\bv$ we define $\Phi(v)$ to be the full subgraph spanned by vertices of $\vG_{A}(S_n)$ that correspond to regions contained in either of the associated discs of $\bv$.

\subsection*{Joins}
Recall that a \emph{join} $X$ is a full subgraph spanned by subsets of vertices $V_1, \dots, V_N$ such that every vertex in $V_i$ spans an edge with every vertex in $V_j$, for all $i \ne j$.  We assume that no $V_i$ is itself a join and we call the set of $V_i$ the \emph{components} of $X$.  We say that a join $X$ is \emph{perfect} if it has at most three components, one or two which consist of infinitely many vertices, and at most one component consisting of a single vertex.  A perfect join $X$ is \emph{maximal} if there are no vertices $v \notin X$ such that $X \cup \{v\}$ spans a perfect join.

\begin{figure}[t]
\centering
\includegraphics[scale=.75]{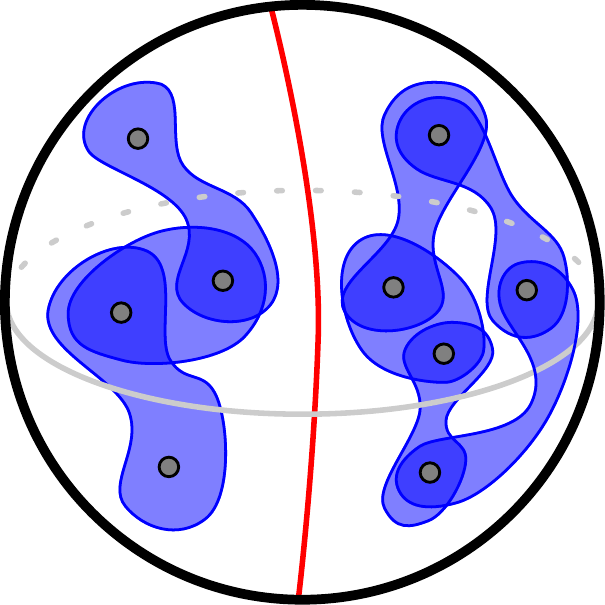}
\caption{Each vertex in $\vC_2(S_9)$ corresponds to a maximal perfect join in $\vG_A(S_9)$, where discs containing two marked points are represented in $A$.}
\label{Joins}
\end{figure}

\begin{lem}\label{partial1}
Let $A \subset \vR(S_n)$ and assume that $\vG_{A}(S_n)$ contains a small vertex. Then the function
\begin{center}$\Phi \ :\  \{ \text{vertices of }\vC_{k}(S_n) \} \rightarrow \{ \text{maximal perfect joins in }\vG_{A}(S_n)\}$\end{center}
is a bijection.
\end{lem}

\begin{proof}
We first show that for any vertex $v \in \vC_k(S_n)$ the subgraph $\Phi(v)$ is a maximal perfect join.  If $v$ corresponds to the curve $\bv$ then the large associated disc of $\bv$ contains infinitely many regions represented in $A$.  If the small associated disc contains infinitely many vertices then $\Phi(v)$ is a perfect join. In this case two components contain infinitely many vertices and a possible third component contains a single vertex corresponding to the annulus whose boundary components are isotopic to the curve $\bv$.  Any vertex $u \in \vG_A(S_n)$ which is not in $\Phi(v)$ must correspond to a region that intersects both associated discs of $\bv$.  Hence $\Phi(v) \cup \{u\}$ is not a perfect join and so $\Phi(v)$ is maximal.  If the small associated disc of $\bv$ contains a single subsurface represented in $A$, then again $\Phi(v)$ is perfect with one component containing infinitely many vertices and one component containing the vertex corresponding to either the small associated disc, or its annular boundary.  As before, any vertex not in $\Phi(v)$ must intersect at least one vertex in each component.  It follows that $\Phi(v)$ is a maximal perfect join.

We will now show that every maximal perfect join is contained in the image of $\Phi$.  Suppose $X$ is a maximal perfect join with two components containing infinitely many vertices.  The vertices in these components correspond to vertices contained in two connected regions $R_1$ and $R_2$.  Suppose $R_1$ is not a disc.  Then there exists a complementary disc $D$ of $R_1$ that does not contain $R_2$.  We can therefore find a vertex of $\vG_A(S_n)$ that is not in $X$ yet spans an edge with every vertex in one of the components of $X$, hence $X$ is not maximal.  Up to relabeling we have shown that both $R_1$ and $R_2$ are discs. It is clear that the discs $R_1$ and $R_2$ are disjoint up to homotopy.  The subsurface $Q:=S_n \setminus \{R_1, R_2\}$ is an annulus. Indeed, otherwise there exists a region represented in $A$ that intersects $R_1$ and is disjoint from $R_2$ and the existence of such a region contradicts the maximality of $X$.  We have therefore shown that any maximal perfect join of this type is equal to the image of a vertex of $\vC_k(S_n)$.  Namely, $X=\Phi(v)$ where $v$ is the vertex that corresponds to the boundary of the annulus $Q$.

A similar proof shows that maximal perfect joins with a single component consisting of infinitely many vertices belongs to the image of $\Phi$ and the details are left to the reader.
\end{proof}

The previous result says that any automorphism of $\vG_A(S_n)$ induces a permutation of the vertices of $\vC_k(S_n)$.  We now prove that this permutation extends to an automorphism of the graph $\vC_{k}(S_n)$.  We say that two joins are \emph{compatible} if they can be written as $V_1 * V_2$ and $W_1 * W_2$ with $V_1 \subset W_1$.
We say that a subgraph $V$ is \emph{compatible} with a subgraph $W$ if $V$ can be written as $V_1 * V_2$ where $V_1$ is nonempty and $V_1 \subset W$.

\begin{lem}\label{partial2}
Let $A \subset \vR(S_n)$ and assume that $\vG_{A}(S_n)$ has no isolated vertices. Let $v,w$ be vertices of $\vC_{k}(S_n)$. Then $v$ and $w$ are connected by an edge in $\vC_{k}(S_n)$ if and only if $\Phi(v)$ is compatible with $\Phi(w)$.	
\end{lem}

\begin{proof}
If the vertices $v$ and $w$ are connected by an edge in $\vC_{k}(S_n)$ then they correspond to disjoint curves. If $\Phi(v)=V_1 * V_2$ then it follows that up to renumbering, $V_1 \subset \Phi(w)$. Suppose $v$ and $w$ are not connected by an edge in $\vC_{k}(S_n)$. The vertices do not correspond to disjoint curves and so there is no way to choose $V_1$ and $W$ such that $V_1 \subset \Phi(w)$. This is a contradiction and so the result follows.
\end{proof}

The next lemma relates the connectivity of the graph of regions $\vG_A(S_n)$ to the connectivity of the subgraph $\vC_k(S_n)$ where $k$ is the value stated at the beginning of this section.

\begin{lem}\label{BigDaddyHelp}
Let $A \subset \vR(S_n)$ so that $\vG_{A}(S_n)$ has no corks or holes and is connected.
\begin{enumerate}
\item Let $R$ be a region represented in $A$; then a simplex in $\vC_{k}(S_n)$ corresponds to $\partial R$.
\item The complex $\vC_{k}(S_n)$ is connected.
\end{enumerate}
\end{lem}

\begin{proof}
It is clear that $\partial R$ is a union of pairwise disjoint curves, one for each complementary disc of $R$. If $R$ is an annulus then there is a unique vertex in $\vC_{k}(S_n)$ that corresponds to $\partial R$. If $R$ is not an annulus then since the complex $\vG_{A}(S_n)$ has no holes, each complementary disc of $R$ contains a region represented in $A$. It follows that each curve in $\partial R$ has associated discs containing at least $k$ marked points. That is, each curve in $\partial R$ is an $l$-curve for some $l \ge k$.

For the second statement let $u$ and $v$ be vertices of $\vC_{k}(S_n)$. Let $R_0, R_N$ be represented in $A$ such that $u$ corresponds to a curve with associated disc containing $R_0$ and $v$ corresponds to a curve with associated disc containing $R_N$. Now let $a_0,\dots,a_N$ be a path in $\vG_{A}(S_n)$ where the vertices $a_0$ and $a_N$ correspond to the regions $R_0$ and $R_N$ respectively. It follows from the first statement that there exist simplices $\sigma_i$ in $\vC_{k}(S_n)$ corresponding to the vertices $a_i$ of $\vG_{A}(S_n)$. Furthermore, since the vertices $a_i$ and $a_{i+1}$ are adjacent in $\vG_{A}(S_n)$ it follows that $\sigma_i \cup \sigma_{i+1}$ spans a simplex of $\vC_{k}(S_n)$. This implies that the vertices $u$ and $v$ are connected via a path contained in the simplices $\sigma_0 , \dots , \sigma_N$.
\end{proof}

Before completing the proof of the main theorem we note that there is a partial order on vertices of $\vG_{A}(S_n)$. We say that $u \preceq v$ if the link of $v$ is contained in the link of $u$. A vertex is \emph{minimal} when it is minimal with respect to the ordering. We say that a vertex of $\vG_{A}(S_n)$ is \emph{$1$-sided} if it corresponds to a region $R$ such that exactly one of its complementary discs contains a region represented in $A$.

\begin{proof}[Proof of Theorem \ref{BigDaddy}]
Let $\vG_{A}(S_n)$ be a connected graph of regions with a small vertex and no holes or corks. We would like to show that
\[
\eta_A : \Mod(S_n) \rightarrow \Aut \vG_{A}(S_n)
\]
 is an isomorphism. It follows from Lemma \ref{Injectivity} that $\eta_A$ is injective. It remains to show that $\eta_A$ is surjective. We have from Lemmas \ref{partial1} and \ref{partial2} that there is a well-defined map $\partial :\Aut \vG_{A}(S_n) \rightarrow \Aut \vC_{k}(S_n)$ where the image under $\phi \in \Aut \vG_A(S_n)$ of a maximal join determines the automorphism $\partial ( \phi )$. We aim to show that $\partial$ is injective.

Suppose $\partial ( \phi )$ is the identity. Let $v$ be a $1$-sided vertex of $\vG_{A}(S_n)$ corresponding to an annulus, we need to show that $\phi(v)=v$. Suppose $v$ corresponds to an annulus whose boundary components are isotopic to a curve $\bv$. Let $R$ and $Q$ be the small and large associated discs of $\bv$ respectively. We want to find a vertex of $\vC_{k}(S_n)$ corresponding to a curve which is not isotopic to the curve $\bv$. Since $\vG_{A}(S_n)$ is connected it contains a vertex $w$ corresponding to a region contained in $Q$. If $w$ corresponds to an annulus then the desired vertex of $\vC_{k}(S_n)$ corresponds to the unique isotopy class of curve contained in the annulus. If $w$ does not correspond to an annulus then from Lemma \ref{BigDaddyHelp}(1) we can find the desired vertex. We do not consider the case where $w$ corresponds to the disc $Q$ as $\vG_{A}(S_n)$ does not contain corks.

Having found a non-peripheral curve in $Q$ we deduce that there exist vertices of $\vC_{k}(S_n)$ which correspond to curves filling $Q$. Each of these vertices is fixed by $\partial ( \phi )$ by assumption and it follows that $\phi(v)$ corresponds to a region disjoint from $Q$. Since $v$ is a $1$-sided vertex we have that $\phi(v)=v$. It can be shown using a similar argument that if $v$ is a $1$-sided minimal vertex of $\vG_{A}(S_n)$ that does not correspond to an annulus then we can deduce that $\phi(v)=v$.

Now assume that $v$ is any other vertex of $\vG_{A}(S_n)$. Let $D$ be a complementary disc of a region $R$ such that $v$ corresponds to $R$. Since $v$ is not a $1$-sided minimal vertex, $R$ is not an annulus, and $\vG_{A}(S_n)$ does not contain holes we have that there exist vertices that span an edge with $v$ and correspond to regions contained in $D$. We will label the set of all such vertices $X$.

Suppose vertices $u,w \in X$ corresponds to regions $R_u , R_w \subset D$. Let $D_u, D_w$ be the enveloping discs of $R_u$ and $R_w$ respectively.  We will write $u \le w$ if $Q_u \subseteq Q_w$ up to homotopy.  It can be shown that a minimal element with respect to this partial order is a $1$-sided vertex of $\vG_{A}(S_n)$ that either corresponds to an annulus, or is minimal with respect to $\preceq$.
It follows that one of the two conditions below holds:

\begin{enumerate}
\item there exist $1$-sided vertices corresponding to annuli and $1$-sided minimal vertices corresponding to regions that fill the disc $D$, or
\item there exists a $1$-sided vertex of $\vG_{A}(S_n)$ corresponding to the annular boundary of $D$ and no vertices of $\vG_{A}(S_n)$ correspond to non-peripheral subsurfaces of $D$.
\end{enumerate}

Given a vertex $v$ let $Y \subset X$ be the set of all the $1$-sided vertices corresponding to annuli and the $1$-sided minimal vertices that do not correspond to annuli. We conclude that $v$ is the unique vertex of $\vG_{A}(S_n)$ spanning an edge with every vertex in $Y$. We have shown that such vertices are fixed by $\phi$ so $\phi(v)=v$. Hence $\partial$ is injective.

The graph of regions $\vG_{A}(S_n)$ contains a small vertex so $n \ge 3k$. From Lemma \ref{BigDaddyHelp} the graph $\vC_{k}(S_n)$ is connected and by Theorem \ref{etas} the natural homomorphism $\eta_k$ is an isomorphism. The diagram
\[
\begin{tikzcd}
\Mod(S_n) \arrow[hookrightarrow]{rr}{\eta_A} \arrow{dr}[swap]{\eta_k}{\cong} & & \Aut \vG_{A}(S_n) \arrow[hookrightarrow]{dl}{\partial} \\ 
 & \Aut \vC_{k}(S_n) & 
\end{tikzcd}
\]
commutes from the definition of $\partial$ and so both $\eta_A$ and $\partial$ are isomorphisms, completing the proof.
\end{proof}

This proof is analogous to the proof of \cite[Theorem 1.3]{TD}. There is an extra detail to consider when the surface has positive genus. Here, a bijective correspondence is constructed between certain types of joins in a complex of regions associated to a closed surface $\Sigma_{g,0}$ and \emph{dividing sets}, multicurves that separate the surface into two regions.

\section{Normal Subgroups}\label{Subbies}

In this section we prove Theorem \ref{Proppy} which determines the automorphism groups of many normal subgroups of $\mcg(S_n)$ and $\mcg(S_n,p)$. This result then implies Theorem \ref{BraidTHM} and Theorem \ref{HJF} as discussed in Section \ref{ResultsApps}.

We will define a graph of regions $\vG_{N}(S_n)$ related to a normal subgroup $N$ of either $\mcg(S_n)$ or $\mcg(S_n,p)$ and then use Theorem \ref{BigDaddy} to prove the result. The vertices of $\vG_{N}(S_n)$ will correspond to the supports of so-called \emph{basic subgroups} of $N$. We first state some results which arise from the study of the Nielsen-Thurston classification of elements of $\mcg(S_n)$.

A \emph{partial pseudo-Anosov} element of $\mcg(S_n)$ is the image of a pseudo-Anosov element of $\mcg(R)$ under the map $\mcg(R) \rightarrow \mcg(S_n)$ induced by the inclusion of a region $R$ in the surface $S_n$. The region $R$ is called the \emph{support} of the partial pseudo-Anosov element and is unique up to isotopy; see \cite{BLM} for more details.

\subsection*{Pure mapping classes.} Using the terminology of Ivanov in \cite{TeichModGroups} we call an element $f \in \Mod(S_n)$ \emph{pure} if it can be written as a product $f_1 \dots f_k$ where
\begin{enumerate}
\item each $f_i$ is a partial pseudo-Anosov element or a power of a Dehn twist; and
\item if $i \ne j$ then the supports of $f_i$ and $f_j$ have disjoint non-homotopic representatives.
\end{enumerate}

The $f_i$ are called the \emph{components} of $f$. We note that the pure elements of $\Mod(S_n)$ are also elements of $\mcg(S_n)$. Furthermore, while the components of a pure mapping class are not canonical in general, the support of a pure element (the union of the supports of the components) is canonical \cite{BLM}.

\subsection*{Pure subgroups.} We call a subgroup of $\mcg(S_n)$ or $\mcg(S_n,p)$ \emph{pure} if each of its elements is pure. The support of a pure subgroup is well-defined and is invariant under passing to finite index subgroups.

Let $\mathcal B$ be a subset of $m < n$ marked points of a surface $\Sigma_{g,n}^b$. We denote by $\mcg(\Sigma_{g,n}^b,\mathcal B)$ the subgroup of $\mcg(\Sigma_{g,n}^b)$ consisting of elements that induce the trivial permutation of the marked points in $\mathcal B$. If $R$ is a component of the support of a pure subgroup $H$, then there is a well defined \emph{reduction map}
\[
H \rightarrow \mcg(R^\circ,\mathcal B)
\]
where $R^\circ$ is the surface obtained by collapsing each boundary component of $R$ to a marked point and $\mathcal B$ is the set of such points.

Recall that two pseudo-Anosov elements with equal support are independent if their corresponding foliations are distinct. Similarly two partial pseudo-Anosov elements are \emph{independent} if their images under the reduction map are independent pseudo-Anosov elements of $\mcg(R^\circ,\mathcal B)$, otherwise they are \emph{dependent}.

We can now state a key result that follows from many of the ideas given above and in \cite{BLM} and \cite[Lemma 5.10]{TeichModGroups}.

\begin{fact}\label{FACT}
Two elements of a pure subgroup commute if and only if
\begin{enumerate}
\item the supports of their components are pairwise disjoint or equal and
\item in the case where two partial pseudo-Anosov components have equal support $R$ the components are dependent.
\end{enumerate}
In particular if two pure elements commute then all of their nontrivial powers commute and if two pure elements do not commute then all of their nontrivial powers fail to commute.
\end{fact}

The next lemma uses this result to relate the centraliser of a non-abelian pure subgroup $H$ and the complement of one of the components of its support.

\begin{lem}[Brendle-Margalit]\label{AbelianCentral}
Let $H$ be a pure non-abelian subgroup of $\mcg(S_n)$. Then there is a component $R$ of the support of $H$ so that the reduction map $H \rightarrow \mcg(R^\circ,\mathcal B)$ has non-abelian image. For any such $R$, the centraliser of $H$ is supported in the complement of $R$.
\end{lem}

The proof of Lemma \ref{AbelianCentral} is given in \cite[Lemma 6.2]{TD}. Note that as the only marked points discussed in this proof are the collapsed boundary components of subsurfaces, the groups $\pmcg(R^\circ)$ and $\mcg(R^\circ,\mathcal B)$ are equal. This is not the case for spheres with marked points.

In general a subgroup of a normal subgroup $N$ does not have connected support. As, however, we are aiming to define a graph of regions where vertices correspond to subgroups of $N$ this issue must now be addressed.

\begin{lem}[Brendle-Margalit]\label{Complement}
Let $N$ be a pure normal subgroup of $\mcg(S_n)$ or $\mcg(S_n,p)$ and let $G$ be a finite index subgroup of $N$.

Let $f$ be an element of $G$ and let $R$ be a region of $S_n$ so that some component of $f$ has support contained as a non-peripheral subsurface of $R$ and all other components have support that is either contained in or disjoint from $R$. Let $J$ be the subgroup of $G$ consisting of all elements supported in $R$. Then
\begin{enumerate}
\item the subgroup $J$ is not abelian;
\item the subgroup $J$ contains an element with support $R$; and
\item the centraliser $C_G(J)$ is supported in the complement of $R$.
\end{enumerate}
\end{lem}

The proof of this lemma depends upon the fact that $N$ is pure. If $f$ is an element of $N$ with a component whose support is a non-annular region $R$ then the \emph{commutator trick} allows us to find an element of $N$ whose support is $R$. The proof is detailed in \cite[Section 6.2]{TD} along with a deeper discussion of the \emph{commutator trick}.

\subsection{Basic Subgroups}

Recall from Section \ref{INTRO} that for a region $R$ we call the disc $D_R$ an enveloping disc of $R$ if $R \subset D$ and for all discs $D$ such that $R \subset D$ we have $n(D_R) \le n(D)$. We say a mapping class is small in $\mcg(S_n)$ if its support is a region $R$ such that $3 n(D_R) -1 \le n$. Similarly, we say a mapping class is small in $\mcg(S_n,p)$ if it is small in $\mcg(S_n)$ and it has support $R$ such that $p \notin R$.

We turn our attention to the \emph{basic subgroups} of a pure subgroup $N$ that is normal in $\mcg(S_n)$ or $\mcg(S_n,p)$. Brendle-Margalit originally defined subgroups of this type for normal subgroups of $\mcg(\Sigma_{g,0})$ \cite[Section 6]{TD}. The supports of these subgroups will correspond to the vertices of the graph of regions we construct in the next section. It follows then that we need to show that the supports are regions of $S_n$, the group $\Mod (S_n)$ acts on the set of supports, and that there is a small \emph{basic subgroup} of $N$. For technical reasons we define a \emph{basic subgroup} of a finite index subgroup $G$ of $N$.

\subsection*{Basic subgroups.} There exists a strict partial order on subgroups of $G$ as follows:
\begin{center}$H \prec H'$ \ \ \ if \ \ \  $C_G(H') \subsetneq C_G(H)$.\end{center}
We say a subgroup of $G$ is a \emph{basic subgroup} if among all non-abelian subgroups of $G$ it is minimal with respect to the partial order.

The next lemma tells us that the supports of basic subgroups are indeed suitable candidates from which we can build a graph of regions.

\begin{lem}[Brendle-Margalit]\label{Action}
Let $N$ be a pure normal subgroup of $\mcg(S_n)$ or $\mcg(S_n,p)$ that contains a small element and let $G$ be a finite index subgroup of $N$.

\begin{enumerate}
\item The support of a basic subgroup of $G$ is a non-annular region of $S_n$.
\item If $B$ is a basic subgroup of $N$ then $B \cap G$ is a basic subgroup of $G$; similarly, any basic subgroup of $G$ is a basic subgroup of $N$.
\item $N$ contains a small basic subgroup.
\end{enumerate}
\end{lem}

\begin{proof}
To prove the first statement we let $H$ be a basic subgroup of $G$. The support of $H$ is not empty and is not an annulus as $H$ is not abelian by the definition of basic subgroup. Suppose the support of $H$ is the entire sphere $S_n$. From Lemma \ref{AbelianCentral} it follows that $C_G(H)$ is trivial. Now, since $G$ is finite index it contains a small element $f$. From Lemma \ref{Complement} we can find a disc $D$ with $n \ge 3n(D) -1$ and a non-abelian subgroup $J$ containing an element $j$ with support $D$ and with centraliser $C_G(J)$ containing an element of the form $hjh^{-1}$. Since $J$ is not trivial this contradicts the fact that $H$ is basic, thus the support of $H$ is not the entire sphere $S_n$.

In order to finish the proof of the first statement we need to show that the support of $H$ is connected, hence a region. Assume then that the support of $H$ is not connected. By Lemma \ref{AbelianCentral} there exists a component $R$ of the support of $H$ so that the image of the reduction map $H \rightarrow \mcg(R^\circ,\mathcal B)$ is not abelian. There must then be an element of $H \subset G$ with a component whose support is a non-peripheral subsurface of $R$. This satisfies the conditions of Lemma \ref{Complement} so the subgroup $J \subset G$ consisting of all elements supported in $R$ is not abelian. By Lemma \ref{AbelianCentral} the centraliser $C_G(H)$ is supported in the complement of $R$, hence $C_G(H) \subseteq C_G(J)$. In order to contradict our assumption that the support of $H$ is not connected we will show that there is an element of $C_G(J) \setminus C_G(H)$.

Let $h$ be an element of $H$ whose support is not contained in $R$. There exists a non-annular region $Q$ disjoint from $R$ that contains a component of the support of $h$ as a non-peripheral subsurface. Once again, this satisfies Lemma \ref{Complement} so there exists an element of $C_G(J)$ not in $C_G(H)$. This contradicts the minimality of the basic subgroup $H$, so the support of $H$ must be connected. This completes the proof of the first statement.

The proof of the second statement is given in \cite[Lemma 6.4]{TD}.

Now we will prove the third statement. We have that $G$ contains a non-trivial pure element $f$ with a small component. Let the support of this small component be a subsurface of the disc $D$ where $n \ge 3n(D)$. We can find a region $R_1$ contained in $D$ satisfying the conditions of Lemma \ref{Complement}. Let
\begin{center}$J_{R_1} = \{h \in G \ |\ $ the support of $h$ is contained in $R_1 \}$.\end{center}
Now, from Lemma \ref{Complement}, $J_{R_1}$ is not abelian and $C_G(J_{R_1})$ is precisely the elements of $G$ with support disjoint from $R_1$. If we can show that $J_{R_1}$ contains a basic subgroup then the result follows from the second statement of this lemma. We assume that $J_{R_1}$ is not itself minimal, so it contains a non-abelian subgroup $J'_{R_1} \prec J_{R_1}$. By Lemma \ref{AbelianCentral} the support of $J'_{R_1}$ has a component that is a subsurface $R_2$ which is non-peripheral in $R_1$. We define $J_{R_2}$ in the same way as $J_{R_1}$ and we see that $J_{R_2} \prec J_{R_1}$. Repeating this process algorithmically we will arrive at a basic subgroup $H$ with support in $D$, hence $H$ is small.
\end{proof}

We note that if $N$ is normal in $\mcg(S_n)$ then $\Mod(S_n)$ acts on the set of supports of basic subgroups of $N$. If $N$ is normal in $\mcg(S_n,p)$ but not normal in $\mcg(S_n)$ or $\Mod(S_n)$ then $\Mod(S_n)$ may not act on the set of supports of basic subgroups.

\subsection{Action on a Graph of Regions}
Let $N$ be a fixed pure normal subgroup of $\mcg(S_n)$ or $\mcg(S_n,p)$. We will now define a graph of regions $\vG_{N}(S_n)$ such that $\Comm N$ has a natural action on $\vG_{N}(S_n)$. We first define $\mathcal G^\sharp_N(S_n)$ to be the graph of regions whose vertices correspond to be the $\Mod(S_n)$-orbit of the supports of the basic subgroups of $N$. From Lemma \ref{Action}(1) we have that the graph of regions $\mathcal G^\sharp_N(S_n)$ has no cork pairs. From Lemma \ref{Action}(3) we have that $\mathcal G^\sharp_N(S_n)$ contains a small vertex. This graph may be disconnected and may contain holes.

Recall that if $v$ is a hole corresponding to the region $R$ then the filling of $v$ is the union of $R$ and the complementary discs with no representatives of subregions in $A$. If a graph of regions $\vG_{A}(S_n)$ has holes then we define its \emph{filling} to be the graph defined by replacing holes with vertices corresponding to their fillings.

We define $\mathcal G^\flat_N(S_n)$ to be the filling of $\mathcal G^\sharp_N(S_n)$. By \cite[Lemma 2.4]{TD} the graph $\mathcal G^\flat_N(S_n)$ has no holes, no corks, and contains a small vertex. By \cite[Lemma 6.5]{TD} the small vertices of $\mathcal G^\flat_N(S_n)$ lie in the same connected component of the graph. We define this connected component to be the graph of regions $\vG_{N}(S_n)$. It is easy to check that since $\mathcal G^\flat_N(S_n)$ has no holes and no corks the graph $\vG_{N}(S_n)$ has no holes and no corks. We have therefore proven the following fact.

\begin{prop}\label{PRP1}
Let $N$ be a pure normal subgroup of $\mcg(S_n)$ or $\mcg(S_n,p)$ that contains a small element. Then the natural map
\[
\Mod(S_n) \rightarrow \Aut \vG_{N}(S_n)
\]
is an isomorphism.
\end{prop}

We will now define the action of $\Comm N$ on the graph of regions $\vG_{N}(S_n)$. To that end, for a basic subgroup $B$ we define $v_B$ to be the vertex corresponding to the support of $B$.

\begin{prop}\label{PRP2}
Let $N$ be a pure normal subgroup of $\mcg (S_n)$ or $\mcg(S_n,p)$ that contains a small element. There is a well defined homomorphism
\[
\Comm N \rightarrow \Aut \vG_{N}(S_n).
\]
\end{prop}

\begin{proof}
We define the homomorphism as follows: if $\alpha : G_1 \rightarrow G_2$ is an isomorphism between finite index subgroups of $N$ and $\alpha_\star$ is the image in $\Aut \vG_{N}(S_n)$ of the equivalence class of $\alpha$, then for any basic subgroup $B$ of $N$ we have
\[
\alpha_\star(v_B)=v_{\alpha(B \cap G_1)}.
\]
The fact that this makes sense is a consequence of Lemma \ref{Action}(2). That is, $B \cap G_1$ is a basic subgroup of $G_1$ and since $\alpha$ is an isomorphism $\alpha (B \cap G_1)$ is a basic subgroup of $G_2$, and is therefore a basic subgroup of $N$. A consequence of \cite[Lemma 6.8]{TD} is that if $N$ is a normal subgroup of $\mcg(S_n)$ then $\Mod(S_n)$ acts on the set of supports of basic subgroups of $N$. In other words the set of these supports are a subset of $\vR(S_n)$.

If $N$ is not a normal subgroup of $\mcg(S_n)$ or $\Mod(S_n)$ then there may exist vertices of $\vG_N (S_n)$ that do not correspond to the supports of basic subgroups.  If $v$ is a vertex of this type and corresponds to a region $R$ we note that $p \in D_R$.  If $p \in R$ then every vertex in the link of $v$ must correspond to the support of a basic subgroup of $N$.  Suppose now that $p \notin R$.  Since $\vG_N (S_n)$ has no holes and no corks it must be that the complementary disc of $R$ that contains $p$ also contains a region $Q$ such that $p \in Q$.  We refer to a vertex $v$ as admissible if it corresponds to a region $R$ such that; either $R$ is the support of a basic subgroup, or $p \in R$.

Now, each vertex in $\vG_N (S_n)$ is completely determined by its link and so we may extend $\alpha_\star$ to all admissible vertices.  Moreover, every vertex that is not admissible is uniquely determined by the subset of admissible vertices in its link.  As above we may now extend $\alpha_\star$ to all vertices of $\vG_N (S_n)$.

To show that the homomorphism is well defined we must show that if $B$ is a finite index subgroup of $N$ and $\alpha : G_1 \rightarrow G_2$ and $\alpha' : G_1' \rightarrow G_2'$ are representatives of the same element of $\Comm N$ then
\[
v_\alpha(B \cap G_1) = v_{\alpha'}(B \cap G_1').
\]
This is true due to \cite[Proposition 6.8]{TD}. Furthermore, this result tells us that this permutation of vertices of $\vG_N(S_n)$ extends to an automorphism of the subgraph of vertices corresponding to supports of basic subgroups. By definition of $\alpha_\star(v)$ for any other vertex $v$ we have that $\alpha_\star$ extends to a automorphism of $\vG_N(S_n)$.
\end{proof}

\subsection{Proof of Theorem \ref{Proppy}}

For any mapping class $f \in \Mod(S_n)$ we denote by $\alpha_f$ the automorphism of $\Mod(S_n)$ given by conjugation by $f$. If $f$ belongs to the normaliser of $N$ then we may consider $\alpha_f$ as an element of $\Aut N$. If there is a restriction of $\alpha_f$ that is an isomorphism between finite index subgroups of $N$ then we can think of the equivalence class $[\alpha_f]$ as an element of $\Comm N$. For $f \in \Mod(S_n)$ let $f_\star$ be its image in $\Aut \vG_N(S_n)$ by the isomorphism in Propoosition \ref{PRP1}. Throughout this section we will use the fact that $f_\star (v_B) = v_{f B f^{-1}}$. A proof of this fact can be found in \cite[Lemma 6.7(4)]{TD}.

\begin{proof}[Proof of Theorem \ref{Proppy}]
Let $N$ be a normal subgroup of $\mcg(S_n)$ or $\mcg(S_n,p)$. Let $\vM_1$ be the normaliser of $N$ in $\Mod(S_n)$. One can see that $\vM_1$ is equal to one of $\mcg(S_n,p)$, $\Mod(S_n,p)$, $\mcg(S_n)$ or $\Mod(S_n)$. Let $P$ be a pure normal subgroup of finite index in $\vM_1$. A subgroup of this type always exists, as shown by Ivanov in \cite{TeichModGroups}. We can define the following sequence of homomorphisms:
\[
\vM_1 \xrightarrow[]{\Phi_1} \Aut N \xrightarrow[]{\Phi_2} \Comm N \xrightarrow[]{\Phi_3} \Comm N \cap P \xrightarrow[]{\Phi_4} \Aut \vG_{N \cap P} (S_n) \xrightarrow[]{\Phi_5} \Mod(S_n).
\]
Where;
\begin{itemize}
\item The map $\Phi_1$ is defined so that $\Phi_1(f) = \alpha_f$,
\item the map $\Phi_2$ sends an element of $\Aut N$ to its equivalence class in $\Comm N$,
\item the map $\Phi_3$ sends an element of $\Comm N$ to the equivalence class of any restriction that is an isomorphism between finite index subgroups of $N \cap P$,
\item the map $\Phi_4$ is defined in Proposition \ref{PRP2}, and
\item the map $\Phi_5$ is defined in Proposition \ref{PRP1}.
\end{itemize}
\noindent We claim that each $\Phi_i$ is injective and that
\[
\Phi_5 \circ \Phi_4 \circ \Phi_3 \circ \Phi_2 \circ \Phi_1 = \Id_{\vM_1}.
\]
We show that $\Phi_1$ is injective by using similar methods to that of Lemma \ref{Injectivity}. If $f \in \vM_1$ commutes with an element $h \in N$ then the reduction system of $h$ is fixed by $f$. There exists an element of $N$ with nonempty reduction system. Elements of the $\vM_1$-orbit of this reduction system fills the large associated disc of an arbitrary $2$-curve. Similar to Lemma \ref{Injectivity} the injectivity of $\Phi_1$ follows.

Now let $\alpha \in \Aut N$ be an element of the kernel of $\Phi_2$ and let $h \in N$ be pseudo-Anosov. Since $\Phi_2(\alpha)$ is the identity the restriction of $\alpha$ to a finite index subgroup $G$ of $N$ is also the identity. There is some $m>0$ so that both $h^m$ and $(fhf^{-1})^m$ lie in $G$ and so 
\[
fh^mf^{-1} = (fhf^{-1})^m = \alpha((fhf^{-1})^m) = \alpha(f)\alpha(h^m)\alpha(f)^{-1} = \alpha(f) h^m \alpha(f)^{-1}.
\]
It follows that $f^{-1}\alpha(f)$ commutes with $h^m$ and so fixes the unstable foliation of $h$. Since $h$ was arbitrary and $N$ is normal in $\vM_1$ it follows that $f^{-1}\alpha(f)$ fixes all $2$-curves. Again, this implies that $f^{-1}\alpha(f)$ is the identity and so $\alpha(f)=f$ for all $f \in N$.

Now, the map $\Phi_3$ is an isomorphism as a finite index subgroup of $N \cap P$ is a finite index subgroup of $N$.

To show that $\Phi_4$ is injective we first fix an isomorphism $\alpha : G_1 \rightarrow G_2$ that represents an element $[\alpha]$ of $\Comm N \cap P$. We denote the image of $[\alpha]$ in $\Aut \vG_{N \cap P}$ by $\alpha_\star$ as in Proposition \ref{PRP2}. We assume that $\alpha_\star$ is the identity. We will show that $G_1=G_2$, hence $[\alpha]$ is the identity. Let $h \in G_1$. Since the natural map $\Mod(S_n) \rightarrow \Aut \vG_{N \cap P}(S_n)$ is an isomorphism it suffices to show that the images of $\alpha(h)$ and $h$ are equal. We denote these images $\alpha(h)_\star$ and $h_\star$ respectively. 

Let $B$ be a basic subgroup of $N$. By Lemma \ref{Action} (2) we can assume that $B \subset G_1$. It follows that
\[
h_\star(v_B) = v_{hBh^{-1}} = \alpha_\star v_{hBh^{-1}} = v_{\alpha(hBh^{-1})} = v_{\alpha(h)\alpha(B)\alpha(h)^{-1}}
\]
\[
= \alpha(h)_\star v_{\alpha(B)} = \alpha(h)_\star \alpha_\star v_B = \alpha(h)_\star v_B.
\]
Since every element of $\Aut \vG_N(S_n)$ is uniquely determined by the image of vertices of the form $v_B$ we have that $h_\star = \alpha(h)_\star$ as desired.

As mentioned previously the map $\Phi_5$ is an isomorphism due to Proposition \ref{PRP1}. By construction the composition is the identity on $\vM_1$ and the claim follows.  If image of $\Phi_5 \circ \Phi_4 \circ \Phi_3$ is equal to $\vM_1$ then the result follows from the fact that each homomorphism is injective. That is, 
\[
\Aut N \cong \Comm N \cong \vM_1.
\]
We now consider the cases where $\Comm N$ is not isomorphic to $\vM_1$. Let $\vM_2$ be the image of $\Phi_5 \circ \Phi_4 \circ \Phi_3$, hence $\vM_1 \subset \vM_2$. Recall that $\Phi_5^{-1} = \eta_{N \cap P}$. We can now construct the following commutative diagram.
\[
\begin{tikzcd}
 & \vM_2 \arrow{ddd}{\Phi_6} \arrow[hookleftarrow]{dr} &  \\
 \eta_{N \cap P}(\vM_2) \arrow{ur}{\Phi_5}[swap]{\cong} &  & \vM_1 \arrow[hookrightarrow]{d}{\Phi_1} \\
 \Comm N \cap P \arrow[hookrightarrow]{u}{\Phi_4}&  & \Aut N \\
 & \Comm N \arrow{ul}{\Phi_3}[swap]{\cong} \arrow[hookleftarrow]{ur}[swap]{\Phi_2} & 
\end{tikzcd}
\]
where $\Phi_6$ is the natural homomorphism. We will show that $\Phi_6$ is the left inverse of $\Phi_5 \circ \Phi_4 \circ \Phi_3$, that is,
\[
\Phi_6 \circ \Phi_5 \circ \Phi_4 \circ \Phi_3 ([\alpha]) = [\alpha].
\]
We fix an isomorphsim $\alpha : H_1 \to H_2$ between finite index subgroups that represents an element $[\alpha] \in \Comm N \cap P \cong \Comm N$. As usual we denote by $\alpha_\star$ the image of $[\alpha]$ in $\eta_{N\cap P}(\vM_2)$. Assume that $\Phi_5 ( \alpha_\star ) = f \in \vM_2$. We need to show that $[\alpha]$ is equal to the restriction of $\alpha_f$, the conjugation map defined by $f$.

To that end, let $h \in H_1$. We want to show that $\alpha(h) = fhf^{-1}$. For any $j \in \vM_2$ let $j_\star$ be its image in $\eta_{N \cap P}(\vM_2)$. In particular, we have that $\alpha_\star = f_\star$. Since $\Phi_5$ is an isomorphism it suffices to show that $(fhf^{-1})_\star = \alpha(h)_\star$.

Let $B$ be a basic subgroup of $N$. Without loss of generality we assume that $B$ is contained in $H_1$. We now have
\begin{align*}
(fh)_\star \big ( v_B \big )&= f_\star h_\star \big ( v_B \big )= \alpha_\star h_\star \big (v_B \big ) = \alpha_\star \big ( v_{h B h^{-1}} \big ) = v_{\alpha(hBh^{-1})}
\\ &=v_{\alpha(h) \alpha(B) \alpha(h^{-1})} = \alpha(h)_\star \big ( v_{\alpha(B)} \big ) = \alpha(h)_\star \alpha_\star \big ( v_B \big )
\\ &= \alpha(h)_\star f_\star \big ( v_B \big ) = (\alpha(h)f)_\star \big ( v_B \big ).
\end{align*}

It follows that $(fh)_\star = (\alpha(h)f)_\star$, that is, $\alpha(h) = fhf^{-1}$, as required. Thus, we have shown that $\Phi_6$ is a left inverse of $\Phi_5 \circ \Phi_4 \circ \Phi_3$. Since the diagram commutes however, we have that it is an isomorphism.  We can use this to prove that $\Phi_1$ is surjective, hence an isomorphism.  Let $\alpha \in \Aut N$, denote $\Phi_2(\alpha)$ by $[\alpha]$, and let $f = \Phi_5 \circ \Phi_4 \circ \Phi_3 ([\alpha])$. From the argument above we have that $\Phi_6(f) = [\alpha]$, that is, $[\alpha_f] = [\alpha]$. We will now show that $\alpha_f = \alpha$. Since $\vM_1$ is the normaliser of $N$ and the diagram commutes, this implies that $\alpha$ belongs to the image of $\Phi_1$ and that it is surjective.

We want to show that $\alpha(j)=fjf^{-1}$ for all $j \in N$. Since $[\alpha]=[\alpha_f]$, there is a finite index subgroup $G$ of $N$ such that the restriction of $\alpha$ to $G$ agrees with the restriction of $\alpha_f$ to $G$. Let $h$ be a pseudo-Anosov element of $N$. There is some $m > 0$ such that $h^m$ and $jh^mj^{-1}$ lie in $G$. We have that $\alpha(h^m)=fh^mf^{-1}$ and $\alpha(jh^mj^{-1})=fjh^mj^{-1}f^{-1}$ and so
\[
f j h^m j^{-1} f^{-1} = \alpha( j h^m j^{-1} ) = \alpha(j) \alpha(h^m) \alpha(j)^{-1} = \alpha(j) f h^m f^{-1} \alpha(j)^{-1}.
\]
Hence
\[
(f^{-1} \alpha(j)^{-1} f j ) h^m = h^m (f^{-1} \alpha(j)^{-1} f j ),
\]
and therefore $f^{-1} \alpha(j)^{-1} f j$ fixes the unstable foliation of $h$. As above, this implies that $f^{-1} \alpha(j)^{-1} f j$ is the identity. So $\alpha(j) = f j f^{-1}$, as desired.
\end{proof}

\bibliography{braidbib}{}
\bibliographystyle{plain}

\end{document}